\documentclass[a4paper,11pt]{article}
\usepackage{graphicx}
\usepackage{amsmath,amsthm,amssymb,enumerate}
\usepackage{euscript,mathrsfs}
\usepackage{dsfont}
\usepackage{xcolor}
\usepackage{empheq}
\usepackage{fancyhdr}
\usepackage{geometry}
\usepackage{tcolorbox}

\usepackage{color}
\catcode`\@=11 \@addtoreset{equation}{section}

 \usepackage[colorlinks=true, linkcolor=blue, citecolor=red, urlcolor=blue]{hyperref}

\usepackage{stmaryrd}
\allowdisplaybreaks

\usepackage[labelformat=simple]{subcaption}

\usepackage{enumitem}
\usepackage{url}

\usepackage{tikz}
\usetikzlibrary{patterns}
\usepackage[normalem]{ulem}
\usepackage{cancel}

\newtheorem{Theorem}{Theorem}[section]
\newtheorem{Proposition}[Theorem]{Proposition}
\newtheorem{Lemma}[Theorem]{Lemma}

\theoremstyle{definition}
\newtheorem{Definition}[Theorem]{Definition}

\newtheorem{Remark}[Theorem]{Remark}

\newcommand{\up}{{\rm up}}
\newcommand{\Up}{{\rm Up}}
\newcommand{\Fup}{{\rm F}_h^{\eps}}

\newcommand{\vrh}{\vr_h}
\newcommand{\vuh}{\vu_h}

\newcommand{\vB}{\mathbf{B}}
\newcommand{\vC}{\mathbf{C}}

\newcommand{\vBh}{\vB_h}
\newcommand{\vCh}{\vC_h}

\newcommand{\p} {\partial}

\newcommand{\avg}[1]{ \left\{\hspace{-0.3em}\left\{ #1 \right\}\hspace{-0.3em}\right\}_\sigma }

\newcommand{\avc}[1]{ \widehat{ #1 } }

\newcommand{\co}[2]{{\rm co}\{ #1 , #2 \}}
\newcommand{\Ov}[1]{\overline{#1}}

\newcommand{\auh}{ \widehat{\vuh} }
\newcommand{\avh}{ \widehat{\vh} }

\newcommand{\us}{ u_\sigma }

\newcommand{\vu}{\mathbf{u}}
\newcommand{\bfv}{\mathbf{v}}

\newcommand{\bfn}{\vn}
\newcommand{\nG}{\bfn_{\sigma}}
\newcommand{\vv}{\bfv}

\newcommand{\vh}{\vv_h}

\newcommand{\bfx}{x}
\newcommand{\xx}{\bfx}

\newcommand {\CR} {V_{0,h}}
\newcommand {\ND} {\mathcal{N}_{0,h}}
\newcommand {\Nh} {\mathcal{N}_h}

\newcommand {\FS} {X_h}

\newcommand{\PiQ}{\Pi_Q}
\newcommand{\PiV}{\Pi_V}
\newcommand{\PiN}{\Pi_N}
\newcommand{\PiVv}{\PiV \vv}
\newcommand{\PiW}{\Pi_W}

\newcommand{\Hcurl}{ H({\rm curl}) }

\newcommand{\ds}{\,{\rm d}S(x)}

\newcommand{\bFormula}[1]{
\begin{equation} \label{#1}}
\newcommand{\eF}{\end{equation}}

\newcommand{\grid}{\Oh}
\newcommand{\Oh}{\Omega_h}

\newcommand{\TS}{\Delta t}

\newcommand{\aleq}{\stackrel{<}{\sim}}

\newcommand{\vr}{\varrho}

\newcommand{\vn}{\vc{n}}
\newcommand{\vc}[1]{{\bf #1}}

\newcommand{\Div}{{\rm div}}
\newcommand{\Grad}{\nabla}
\newcommand{\Curl}{{\rm Curl}}

\newcommand{\Gradh}{\nabla_h}
\newcommand{\Divh}{{\rm div}_h}
\newcommand{\Curlh}{{\rm Curl}_h}

\newcommand{\dx}{\,{\rm d} {x}}

\newcommand{\dt}{\,{\rm d} t }

\newcommand{\jump}[1]{\left\llbracket#1\right\rrbracket}
\newcommand{\abs}[1]{\left\lvert#1\right\rvert}

\newcommand{\norm}[1]{\left\lVert#1\right\rVert}

\newcommand{\dxdt}{\dx \dt}
\newcommand{\intO}[1]{\int_{\Omega} #1 \dx }
\newcommand{\intTO}[1]{\int_0^T \int_{\Omega} #1 \dxdt}
\newcommand{\intTOB}[1]{ \int_0^T \int_{\Omega} \left( #1 \right) \dxdt}
\newcommand{\intOB}[1]{  \int_{\Omega} \left( #1 \right) \dx}
\newcommand{\Om}{\Omega}
\newcommand{\intK}[1]{\int_{K} #1  \dx}
\newcommand{\intG}[1]{\int_{\sigma} #1 \ds}

\newcommand{\R}{\mathbb{R}}
\newcommand{\I}{\mathbb{I}}
\renewcommand{\S}{\mathbb{S}}

\newcommand{\newcom}{\newcommand}
\newcommand{\beq}{\begin{equation}}
\newcommand{\eeq}{\end{equation}}
\newcom{\ben}{\begin{eqnarray}}
\newcom{\een}{\end{eqnarray}}
\newcom{\beno}{\begin{eqnarray*}}
\newcom{\eeno}{\end{eqnarray*}}
\newcom{\bali}{\begin{aligned}}
\newcom{\eali}{\end{aligned}}

\newcommand{\f}{\frac}

\definecolor{Cgrey}{rgb}{0.85,0.85,0.85}
\definecolor{Cblue}{rgb}{0.50,0.85,0.85}
\definecolor{Cred}{rgb}{1,0,0}
\definecolor{fancy}{rgb}{0.10,0.85,0.10}
\definecolor{forestgreen}{rgb}{0.13, 0.55, 0.13}

\date{}
\newcommand{\pd}{\partial}
\newcommand{\pdt}{\pd _t}
\newcommand{\order}{{\mathcal O}}

\newcommand{\Hc}{\mathcal{H}}
\newcommand{\eps}{\varepsilon}
\newcommand{\faces}{\mathcal{E}}

\newcommand{\facesK}{\faces(K)}

\newcommand{\facesint}{\faces^I}

\newcommand{\facesext}{\faces^B}

\newcommand{\mcE} {\faces}
\newcommand{\mcEe} {\facesext}
\newcommand{\mcEi} {\facesint}
\newcommand{\mcP} {\mathcal{P}}

\begin{document}

\pagestyle{fancy} \lhead{\color{blue}{Convergence of numerical solutions for compressible MHD}}
\rhead{\emph{Y.Li and B.She}}

\title{
\bf{
On convergence of numerical solutions
for the compressible MHD system with weakly divergence-free magnetic field}
}

\author{
Yang Li$^1$ \quad \,\,\,\,\,\,\,\,\,\,\,\,\,\,\,\,\,\,\,\,\,\,\,\,  Bangwei She$^{2,3}$ \\ \\ $^1$School of Mathematical Sciences, \\ Anhui University,  230601, Hefei, People's Republic of China \\ Email: lynjum@163.com\\ \\
$^{2}$Institute of Mathematics, \\ Czech Academy of Sciences, \v Zitn\'a 25, 115 67, Praha 1, Czech Republic \\ and \\
$^3$Department of Mathematical Analysis,  Charles University\\ Sokolovsk\'{a} 83, 186 75, Praha 8, Czech Republic \\
 Email: she@math.cas.cz \\ \\
}

\date{\today}

\maketitle

\begin{abstract}

We study a general convergence theory for the analysis of numerical solutions to a magnetohydrodynamic system describing the time evolution of compressible, viscous, electrically conducting fluids in space dimension $d$ $(=2,3)$.
First, we introduce the concept of dissipative weak solutions and prove the weak--strong uniqueness property for dissipative weak solutions, meaning a dissipative weak solution coincides with a classical solution emanating from the same initial data on the lifespan of the latter. Next, we introduce the concept of consistent approximations and prove the convergence of consistent approximations towards the dissipative weak solution as well as the classical solution. Interpreting the consistent approximation as the energy stability and consistency of numerical solutions, we have built a nonlinear variant of the celebrated Lax equivalence  theorem. Finally, as an application of this theory, we show the convergence analysis of two numerical methods.

\end{abstract}



\tableofcontents

\section{Introduction}
The time evolution of unsteady, electrically conducting fluids in the presence of magnetic field is described by the magnetohydrodynamic (MHD) system. Mathematical theory of MHD is widely applied in astrophysics and thermonuclear reactions, among many others. A simplified and well-accepted model for compressible isentropic MHD system admits the form (see for instance \cite{HW10}):
\begin{equation}\label{pde}
\left\{\begin{aligned}
& \pdt \vr+\Div (\vr \vu)=0,\\
& \pdt (\vr\vu)+\Div (\vr \vu \otimes \vu )+\Grad p(\vr)=\Div \mathbb{S}(\Grad \vu)+\Curl\vB\times \vB,\\
& \pdt \vB=\Curl (\vu \times \vB)- \alpha \Curl  (\Curl \vB),\\
& \Div \vB=0, \\
\end{aligned}\right.
\end{equation}
in the time-space domain $(0,T)\times \Omega, \; \Omega \subset \R^d,\,d=2,3$.  Here, $t\in (0,T)$ and $x\in \Omega$ represent the time and space variables, respectively. We denote by $\vr=\vr(t,x)$ the density of the fluids, $\vu=\vu(t,x)\in \R^d$ the velocity field, $\vB=\vB(t,x)\in \R^d$ the magnetic field and $p=p(\vr)$ the scalar pressure. $\alpha>0$ is the resistivity coefficient acting as the magnetic diffusion. $\mathbb{S}=\mathbb{S}(\Grad \vu)$ stands for the Newtonian viscous stress tensor given by
\[
\mathbb{S}(\Grad \vu)=\mu \left( \Grad \vu+\Grad^T \vu-\f{2}{d}\Div \vu \mathbb{I}  \right)+\lambda \Div \vu \mathbb{I},
\]
where $\mu>0$ and $\lambda \geq 0$ are the shear and bulk viscosity coefficients, respectively. The pressure $p$ is assumed to satisfy the isentropic law
\begin{equation}\label{plaw}
p(\vr)=a \vr^{\gamma},\,\,a >0,
\end{equation}
where $\gamma>1$ is the adiabatic exponent. System \eqref{pde} is supplemented with the boundary conditions ($\vn$ denotes the unit outward normal on the boundary $\pd \Omega$):
\begin{equation}\label{BCs}
\left\{\begin{aligned}
& \text{either non-slip boundary conditions: }\vu|_{\pd \Omega}=\mathbf{0},\quad
 \,\,
\vB\times \vn|_{\pd \Omega}=\mathbf{0},\\
& \text{ or periodic boundary conditions: }\Omega=\mathcal{T}^d=([0,1]_{\{0,1\}})^d,\\
\end{aligned}\right.
\end{equation}
together with the initial conditions:
\begin{equation}\label{ini_c}
 (\vr,\vr\vu,\vB)|_{t=0}=(\vr_0,\mathbf{m}_0,\vB_0).
\end{equation}

The well-posedness of the compressible MHD system~\eqref{pde} has been studied in several occasions. We refer to Vol'pert and Hudjaev \cite{VoHu} for local well-posedness with positive initial density, Fan and Yu \cite{FY} for local well-posedness with initial vacuum. The existence of global weak solutions with finite energy initial data was studied by Hu and Wang \cite{HW10} for $\gamma>\frac{3}{2}$ in three dimensions. Moreover, we refer to \cite{Ka84,LiZh13} for global well-posedness theory with smallness of initial data, either close to equilibrium state or smallness of initial energy but possibly large oscillations.

The convergence analysis of numerical solutions for compressible viscous fluids was first reported by Karper~\cite{Karper} and further studied by Feireisl et al.~\cite{FL18,FLMS_FVNS,FeLMMiSh}. When a magnetic field is coupled to compressible viscous fluids, 
as far as we know, the only result on numerical convergence was done by Ding and Mao~\cite{Ding}. They studied the convergence of a mixed finite volume (FV)-finite element (FE) approximation towards (a suitable subsequence of) weak solutions requiring the technical assumption on the adiabatic exponent $\gamma>3$.  Unfortunately it excludes physically relevant parameters, e.g. $\gamma =7/5$ for the diatomic gas. Therefore, it is  significant to study the case of ``small" $\gamma$ that covers the physical parameters. 

The aim of this paper is to establish a general framework for the convergence analysis of numerical approximations for the compressible MHD system \eqref{pde} for the full range of $\gamma\in(1,\infty)$. As a byproduct, we also prove global solvability to the compressible MHD system for any $\gamma>1$ and large initial data. The strategy is built on the concepts of \emph{dissipative weak solutions} and \emph{consistent approximation}, see respectively Definition~\ref{def_dw} and Definition~\ref{def_ca}.
A dissipative weak (DW) solution allows concentration and oscillation defects that can be controlled by the dissipation defects of the energy stability. It should be stressed that constructing a weak solution for ``small" $\gamma<d/2$ is a challenging task for compressible viscous fluids while the dissipative character of the DW solution allows us to work with ``small" $\gamma$.
Though the DW solution is weaker than the standard finite energy weak solution, it satisfies the weak--strong uniqueness principle, meaning that a DW solution coincides with a classical solution as long as the latter exists. By this argument, the convergence of a numerical solution reduces to the convergence towards a DW solution. Upon realizing a DW solution can be obtained by the limit (discretization parameter $h \to 0$) of a consistent approximation interpreting the stability and consistency properties of the numerical solution,  we find a generalized Lax equivalence  theorem:
\begin{center}
convergence (to a DW or classical solution) $\Longleftrightarrow$ stability + consistency =:\; consistent approximation.
\end{center}
More specifically, our strategy reads:
\begin{itemize}
\item Proving the weak--strong uniqueness principle in the class of DW solution, see Theorem~\ref{Th1}.
\item Passing to the limit ($h \to 0$) from the consistent approximation to construct a DW solution in the sense of Definition~\ref{def_dw}.
\item Showing that a numerical solution is a consistent approximation in the sense of Definition~\ref{def_ca}  that reflects the stability and consistency of the numerical solution.
\end{itemize}

In this paper, the convergence theory is aimed for the class of numerical schemes that preserves the divergence free of the magnetic field weakly. Other properties such as conservation of the mass, positivity of the density, and stability of the total energy are also expected for the numerical solutions.

%
%
%
%
%
%

The rest of the paper is organized as follows.
In Sect.~\ref{sec_main} we first introduce the concepts of DW solutions and consistent approximation. Then we state the main theorems; these are the weak--strong uniqueness property and convergence of a consistent approximation for the compressible MHD system. In Sect.~\ref{dmv_strong} we establish the relative energy inequality in the framework of DW solutions and prove the weak--strong uniqueness principle through the Gronwall-type argument.
In Sect.~\ref{sec_con}, we prove the convergence of a consistent approximation and apply it to the convergence analysis of numerical solutions of two mixed finite volume--finite element methods. The paper ends up with the conclusion.

\section{Main results}\label{sec_main}

\subsection{Preliminaries}
Let $\mathcal{M}\left( \overline{\Om}\right)$ signify the space of signed Borel measures over $ \overline{\Om}$ and let $\mathcal{M}^{+}\left( \overline{\Om}\right)$ be the non-negative ones. Moreover, we recall that $\psi \in L^2_0(\Om)$ means $\varphi \in L^2(\Om)$ with zero mean. We now introduce the concept of DW solutions.

\begin{Definition}[DW solution]
\label{def_dw}
A triple $(\vr,\vu,\vB)$ is said to be a dissipative weak solution to the MHD system \eqref{pde}--\eqref{ini_c} provided that
\begin{itemize}
\item{ Regularity of solution
\[
\vr \geq 0,\,\, \vr \in L^{\infty}(0,T;L^{\gamma}(\Om)),\,\,\sqrt{\vr}\vu \in L^{\infty}(0,T;L^{2}(\Om;\R^d)),
\]
\[
\vB \in L^{\infty}(0,T;L^{2}(\Om;\R^d)),\,\,\Grad \vu, \Curl  \vB \in L^2(0,T;L^{2}(\Om;\R^{d\times d}));
\]
}

\item{ The continuity equation
\beq\label{lbb3}
\int_0^{\tau}\int_{\Om} \Big ( \vr \partial_t \varphi + \vr\vu\cdot \Grad  \varphi  \Big )\dxdt
=\left[ \int_{\Om} \vr \varphi \dx\right]_{t=0}^{t=\tau}
\eeq
for a.e. $\tau \in (0,T)$, any $\varphi \in C_c^1([0,T]\times \overline{\Om})$;
}

\item { The momentum equation
\[
\int_0^{\tau}\int_{\Om} \Big ( \vr\vu\cdot \p_t \vv + \vr\vu\otimes \vu: \Grad  \vv+p(\vr)\Div  \vv-\mathbb{S}(\Grad \vu):\Grad  \vv
+\left(\Curl \mathbf{B} \times \mathbf{B}\right)\cdot \vv
 \Big ) \dxdt
\]
\beq\label{lbb4}
+
 \int_0^{\tau}\int_{\overline{\Om}} \Grad  \vv :{\rm d} \mu_c(t)\dt
 +
 \int_0^{\tau}\int_{\overline{\Om}} \vv \cdot {\rm d} \mu_B(t)\dt
=\left[ \int_{\Om}  \vr\vu \cdot\vv   \dx \right]_{t=0}^{t=\tau}
\eeq
for a.e. $\tau \in (0,T)$, any $\vv \in C_c^1([0,T]\times \Om;\R^d)$ and some $\mu_c\in L^{\infty}(0,T;\mathcal{M}(\overline{\Om};\R^{d \times d}_{sym}))$, $\mu_B\in L^{2}(0,T;\mathcal{M}(\overline{\Om};\R^{d \times d}))$;
}

\item { The Maxwell's equation
\beq\label{lbb5}
\int_0^{\tau}\int_{\Om} \Big (\mathbf{B}\cdot \p_t \vv+ (\vu \times \mathbf{B})\cdot \Curl  \vv
-\alpha  \Curl  \mathbf{B} \cdot \Curl  \vv \Big ) \dxdt
=\left[ \int_{\Om}  \vB\cdot\vv   \dx \right]_{t=0}^{t=\tau}
\eeq
for a.e. $\tau \in (0,T)$, any $\vv \in  C_c^1 ([0,T]\times \overline{\Om};\R^d),\,\vv\times \vn|_{\p \Om}=\mathbf{0}$};

\item { Divergence free of magnetic field
\beq\label{lbb6}
\int_{\Om} \vB \cdot \Grad \varphi\dx
=0
\eeq
for any
${\varphi} \in C^1( \overline{\Om}) \cap L_0^2(\Om) $};

\item { Balance of total energy
\[
\int_{\Om} \left[\f{1}{2}\vr |\vu|^2+\f{1}{2}|\mathbf{B}|^2+\mathcal{H}(\vr)\right](\tau,x)\dx
+\int_0^{\tau}\int_{\Om} \left (  \mathbb{S}(\Grad\vu):\Grad \vu
+\alpha \left|  \Curl  \mathbf{B} \right|^2
\right )\dxdt
\]
\beq\label{lbb7}
+\int_{\overline{\Om} }{\rm d}\mathfrak{D}(\tau)
+ \int_0^{\tau}
  \int_{\overline{\Om}}  {\rm d}\mathfrak{C}
\leq
\int_{\Om}\left[ \f{1}{2}\f{|\mathbf{m}_0|^2}{\vr_0}+\f{1}{2}|\mathbf{B}_0|^2+\mathcal{H}(\vr_0)\right]\dx
\eeq
for a.e. $\tau \in (0,T)$ and some $\mathfrak{D}\in L^{\infty}(0,T; \mathcal{M}^{+}(\overline{\Om}))$, $\mathfrak{C} \in \mathcal{M}^{+}([0,T]\times \overline{\Om})$;
}

\item { Compatibility conditions
\beq\label{lbb8}
 \int_0^T \psi(t)\int_{\overline{\Om}}{\rm d}|\mu_c(t)|\dt \lesssim \int_0^T \psi(t)\int_{\overline{\Om}}{\rm d}\mathfrak{D}(t)\dt,
\eeq
\beq\label{lbb8-2}
 \int_0^T \psi(t)\int_{\overline{\Om}}{\rm d}|\mu_B(t)|\dt  \lesssim  \f{1}{\epsilon}\int_0^T \psi(t)\int_{\overline{\Om}}{\rm d}\mathfrak{D}(t)\dt +\epsilon  \int_0^T
  \int_{\overline{\Om}} \psi(t) {\rm d}\mathfrak{C}
\eeq
for any $\epsilon>0,\psi \in C([0,T]),\psi \geq 0$.
}
\end{itemize}
\end{Definition}
\medskip

\begin{Remark}
In (\ref{lbb4}), the measure $\mu_c$ denotes the oscillation and concentration defects due to the nonlinear terms $\vr \vu\otimes \vu$ and $p(\vr)$, while $ \mu_B$ reflects that of $\Curl \vB \times \vB$. In (\ref{lbb7}), the non-negative measure $\mathfrak{D}$ represents the defects from the total energy $\f{1}{2}\vr |\vu|^2+\f{1}{2}|\mathbf{B}|^2+\mathcal{H}(\vr)$, while $\mathfrak{C}$ means the defects from the dissipative terms $
 \mathbb{S}(\Grad\vu):\Grad \vu+\alpha \left|  \Curl  \mathbf{B} \right|^2$. Furthermore, these measures are interrelated through the compatibility conditions (\ref{lbb8}) and (\ref{lbb8-2}), which play a crucial role in proving the weak--strong uniqueness principle.
\end{Remark}

Next, let us define the concept of consistent approximation in terms of the stability and consistency of a numerical solution.

\begin{Definition}[Consistent approximation]\label{def_ca}
Let the differential operator $\Gradh$ (resp. $\Divh$ and $\Curlh$) be compatible\footnote{A simple example of such compatibility is that $\Gradh = \Grad$ element-wisely, see~\cite[Section 11.4]{FeLMMiSh}.}  with $\Grad$ (resp. $\Div$  and $\Curl$) in the sense of \cite[Definition 5.8]{FeLMMiSh}.
 We say a numerical approximation $(\vrh,\vuh,\vBh)$ of the MHD system~\eqref{pde}--\eqref{ini_c} is a consistent approximation if the following stability and consistency conditions hold:
\begin{enumerate}
\item {\bf Stability.}
The numerical approximation is stable in the sense that
\begin{equation}\label{es}
E_h(\tau) +  \int_0^\tau \intOB{ \S(\Gradh \vuh) : \Gradh \vuh
+  \alpha | \Curl \vBh |^2 } \leq E_h(0),  \quad \forall \tau \in [0,T],
\end{equation}
where $ E_h =   \intO{ \left(\frac{1}{2} \vrh  \abs{\Pi_h \vuh }^2  + \mathcal{H}(\vrh )+\frac{1}{2} \abs{\vBh }^2 \right)  } $  the total energy,  $\Hc(\vrh) =  \frac{a}{\gamma-1}\vrh^\gamma$ the pressure potential, and  $\Pi_h$ is either identity or a piecewise constant projection operator satisfying $\norm{\Pi_h \vuh - \vuh }_{L^2(\Om)} \leq h \norm{\Gradh \vuh}_{L^2(\Om)}$.
\item {\bf Consistency.}

The numerical approximation is consistent if: 

\begin{subequations}\label{cP}
$\bullet$ {\bf Continuity equation.} It holds for  any $\phi \in C_c^1([0,T) \times \Ov{\Omega})$ that
\begin{equation} \label{cP1}
 \intTOB{  \vrh \partial_t \phi + \vrh \vuh \cdot \Grad \phi }
 =- \intO{ \vrh^0 \phi(0,\cdot) }
+  e_{1,h}[\phi],
\end{equation}
where $e_{1,h}[\phi] \to 0 \mbox{ as } h \to 0 \mbox{ for any } \phi \in C_c^M([0,T) \times \Ov{\Omega})  \mbox{ for some integer } M\geq 1;$

$\bullet$ {\bf Balance of momentum.}  It holds for any $\vv \in C_c^1([0,T) \times \Omega; \R^d)$ that
\begin{equation} \label{cP2}
\begin{aligned}
 \intTOB{  \vrh \Pi_h \vuh \cdot \partial_t \vv + \vrh \Pi_h \vuh \otimes \vuh  : \Grad \vv  + p_h \Div \vv
 -    \S( \Gradh \vuh) : \Grad \vv }
\\  +\intTO{ ( \Curlh \vBh \times \vBh ) \cdot  \vv}
= - \intO{ \vrh^0 {\Pi_h \vuh^0} \cdot \vv(0,\cdot) }
+e_{2,h}[\vv]
\end{aligned}
\end{equation}
where $e_{2,h}[\vv] \to 0$ as $ h \to 0$  for any $ \vv \in C_c^M([0,T) \times \Omega; \R^d)$  for some integer  $M\geq 1$;

$\bullet$ {\bf The Maxwell equation} It holds for any $\vC \in C_c^1([0,T) \times \Ov{\Omega}; \R^d)$,  $\vC \times \vn|_{\p \Omega}=\mathbf{0}$ that
 \begin{equation} \label{cP3}
   \intTOB{ \vBh \cdot \pdt \vC + (\vuh \times \vBh -  \alpha  \Curlh \vBh ) \cdot \Curl \vC}
=-  \intO{ \vBh^0 \cdot \vC(0,\cdot)}  +e_{3,h}[\vC]
 \end{equation}
 where $e_{3,h}[\vC] \to 0$ as $ h \to 0$  for any $ \vC \in C_c^M([0,T) \times \Omega; \R^d)$  for some integer  $M\geq 1$;

$\bullet$ {\bf  Weakly divergence free of magnetic field.} It holds for any $\psi  \in C^1(\Ov{\Omega})\cap L^2_0(\Omega)$ that
 \begin{equation}\label{cP4}
 \intO{ \vBh \cdot \Grad \psi} = e_{4,h}[\psi]
 \end{equation}
where $e_{4,h}[\psi] \to 0$ as $ h \to 0$  for any $ \psi \in C^M(\Ov{\Omega})\cap L^2_0(\Omega)$ for some integer  $M\geq 1$.
\end{subequations}
\end{enumerate}
\end{Definition}
\medskip

\subsection{Main theorems}

Our main results in this paper are summarized in the following two theorems. The first one is concerned with the stability of classical solutions within DW solutions.
\begin{Theorem}[weak--strong uniqueness]\label{Th1}
Let $(\vr,\vu,\vB)$ be a DW solution to \eqref{pde}--\eqref{ini_c} with the initial data $(\vr_0,\vr_0\vu_0,\vB_0)$.
Suppose that $(\widetilde{\vr},\widetilde{\vu},\widetilde{\vB})$ is a classical solution to \eqref{pde}--\eqref{ini_c} starting from the same initial data with $\vr_0>0,\Div \vB_0=0$ and belonging to the class
\begin{equation}\label{STC}
 \vr \in C^1([0,T]\times \overline{\Om}),\,\,\,\, \vu,\vB \in
  C^2([0,T]\times \overline{\Om};\R^d).
\end{equation}
Then $\mu_c=\mathbf{0},\,\,\,\, \mu_B=\mathbf{0},\,\,\,\, \mathfrak{D}=0,\,\,\,\, \mathfrak{C}=0$ and
\[
\vr=\tilde{\vr},\,\,\,\, \vu=\tilde{\vu},\,\,\,\,\vB=\tilde{\vB},
\text{    in   } (0,T)\times \Om.
\]
\end{Theorem}
\medskip

The second one gives the convergences of numerical solutions.
\begin{Theorem}[Convergence]\label{Th2}
Let  $(\vrh,\vuh,\vBh)$ be a consistent approximation of the MHD system in the sense of Definition~\ref{def_ca}. Then the following convergences hold:
\begin{enumerate}
\item {\bf Convergence to DW solution.}
There exists a subsequence of $(\vrh,\vuh,\vBh)$ not relabelled such that
\begin{align*}
\vrh \rightarrow & \; \vr \mbox{ weakly-(*) in } L^\infty(0,T;L^\gamma(\Om)),\\
\vuh \rightarrow &  \; \vu \mbox{ weakly in } L^2 ((0,T)\times \Om; \R^d)\\
\vBh \rightarrow &  \;  \vB \text{ weakly-(*) in } L^\infty(0,T; L^2( \Om; \R^d)),
\end{align*}
where the triple $(\vr, \vu,\vB)$ represents a DW solution to the MHD system in the sense of Definition~\ref{def_dw}.
\item {\bf Convergence to classical solution.}
In addition, let the MHD system \eqref{pde}--\eqref{ini_c} admit a classical solution in the class \eqref{STC}. Then the above weak limit is unconditional (no need of subsequence but the whole sequence) and the limit quantity $(\vr,\vu,\vB)$ coincides with the classical solution.
\end{enumerate}
\end{Theorem}
\medskip

\section{Weak--strong uniqueness}\label{dmv_strong}

In this section, we aim to prove the weak--strong uniqueness theory (Theorem~\ref{Th1}) for the DW solutions given in Definition~\ref{def_dw}. To this end, we invoke the relative entropy functional to measure the distance between a DW solution and a classical solution.  For definiteness, we shall proceed in case of Dirichlet boundary conditions, while the periodic case can be carried out exactly in the same way.

\subsection{Relative energy inequality}\label{rela}
The goal of this part is to establish the relative energy inequality in the context of DW solutions.

Let $(\vr,\vu,\vB)$ be a DW solution to (\ref{pde})--(\ref{ini_c}) and $(r,\mathbf{U},\mathbf{b})$ be subject to
\begin{equation*}\label{lbn1}
\left\{\begin{aligned}
& r \in C^1([0,T]\times \overline{\Om}),\,\,\,\,r>0,\\
& \mathbf{U}\in C^1([0,T]\times \overline{\Om};\R^d),\,\,\,\,\mathbf{U}|_{\pd \Om}=\vc{0},\\
& \mathbf{b}\in C^1([0,T]\times \overline{\Om};\R^d),\,\,\,\,\mathbf{b} \times \vn|_{\pd \Om}=\mathbf{0},\,\,\,\,\Div \mathbf{b}=0.\\
\end{aligned}\right.
\end{equation*}
Inspired by \cite{FNS} in the context of finite energy weak solutions, we introduce the \emph{relative entropy} in the framework of DW solutions
\begin{equation*}\label{lbn2}
\mathcal{E}\Big((\vr,\vu,\vB)\,\Big|\,(r,\mathbf{U},\mathbf{b})\Big)(\tau)
=\int_{\Om}\left[\f{1}{2}\vr |\vu-\mathbf{U}|^2+\f{1}{2} |\vB-\mathbf{b}|^2
+\mathcal{H}(\vr)-\mathcal{H}(r)-\mathcal{H}'(r)(\vr-r)
\right](\tau,\cdot) \dx.
\end{equation*}
Notice that we may rewrite the relative entropy in an equivalent form as follows
\begin{equation}\label{lbn3}
\begin{split}
& \mathcal{E}\Big((\vr,\vu,\vB)\,\Big|\,(r,\mathbf{U},\mathbf{b})\Big)(\tau)
=\int_{\Om}\left( \f{1}{2}\vr |\vu|^2+\f{1}{2}|\vB|^2+\mathcal{H}(\vr) \right) \dx
+\int_{\Om}\f{1}{2} \vr |\mathbf{U}|^2\dx
\\& -\int_{\Om}\vr \vu \cdot \mathbf{U} \dx
-\int_{\Om} \vB\cdot \mathbf{b} \dx
-\int_{\Om} \vr \mathcal{H}'(r) \dx
+\int_{\Om} p(r) \dx+\f{1}{2}\int_{\Om}|\mathbf{b}|^2  \dx.
\end{split}
\end{equation}
The crucial observation is that the integrals on the right-hand side of (\ref{lbn3}) can be expressed through (\ref{lbb3})--(\ref{lbb7}) with suitable choices of test functions. To handle the density-dependent terms, we first test the continuity equation (\ref{lbb3}) with $\f{1}{2}|\mathbf{U}|^2$ to derive
\beq\label{lbn4}
\left[ \int_{\Om} \f{1}{2} \vr  |\mathbf{U}|^2  \dx \right]_{t=0}^{t=\tau}
=\int_0^{\tau}\int_{\Om} \Big ( \vr \mathbf{U}\cdot \pdt \mathbf{U}+ \vr\vu \cdot \Grad \mathbf{U} \cdot \mathbf{U} \Big )\dxdt.
\eeq
Moreover, we take $\mathcal{H}'(r)$ as a test function in (\ref{lbb3}) to find
\beq\label{lbn5}
\left[ \int_{\Om}  \vr \mathcal{H}'(r)   \dx \right]_{t=0}^{t=\tau}
=\int_0^{\tau}\int_{\Om} \Big ( \vr\pdt \mathcal{H}'(r)+ \vr\vu \cdot  \Grad \mathcal{H}'(r) \Big )\dxdt.
\eeq
Upon choosing $\mathbf{U}$ as a test function of the momentum equation (\ref{lbb4}), we observe that
\[
\left[ \int_{\Om}  \vr\vu \cdot \mathbf{U}  \dx \right]_{t=0}^{t=\tau}
=\int_0^{\tau}\int_{\Om} \Big ( \vr\vu \cdot \pdt \mathbf{U} + \vr\vu\otimes \vu : \Grad \mathbf{U}+p(\vr)\Div \mathbf{U}
\]
\beq\label{lbn6}
-\mathbb{S}(\Grad\vu ):\Grad \mathbf{U}
+\left(\Curl  \vB \times \vB \right)\cdot \mathbf{U}
 \Big ) \dxdt
 +
 \int_0^{\tau}\int_{\overline{\Om}} \Grad  \mathbf{U} :{\rm d} \mu_c(t)\dt
 +
 \int_0^{\tau}\int_{\overline{\Om}} \mathbf{U} \cdot {\rm d} \mu_B(t)\dt.
\eeq
Next, to calculate the term involved with the magnetic field, we choose $\mathbf{b}$ as a test function in (\ref{lbb5}) to deduce that
\beq\label{lbn7}
\left[ \int_{\Om}  \vB \cdot\mathbf{b}  \dx \right]_{t=0}^{t=\tau}
=\int_0^{\tau}\int_{\Om} \Big (\vB \cdot \pdt \mathbf{b}+ (\vu \times \vB) \cdot \Curl  \mathbf{b}
-\alpha\Curl  \vB  \cdot \Curl  \mathbf{b} \Big ) \dxdt.
\eeq
Finally, combining (\ref{lbn4})--(\ref{lbn7}) with the balance of total energy (\ref{lbb7}), we obtain the \emph{relative energy inequality} as follows
\[
\left[\mathcal{E}\Big((\vr,\vu,\vB)\,\Big|\,(r,\mathbf{U},\mathbf{b})\Big)\right]_{t=0}^{t=\tau}
+\int_0^{\tau}\int_{\Om}   \mathbb{S}(\Grad\vu-\Grad \mathbf{U}):(\Grad \vu-\Grad \mathbf{U})
\dxdt
\]
\[
+\alpha \int_0^{\tau}\int_{\Om} \left|  \Curl  (\vB -\mathbf{b})\right|^2\dxdt
+\int_{\overline{\Om} }{\rm d}\mathfrak{D}(\tau)
+ \int_0^{\tau}
  \int_{\overline{\Om}}  {\rm d}\mathfrak{C}
\]
\[
\leq -\int_0^{\tau}\int_{\Om} \Big ( \vr\vu \cdot \pdt \mathbf{U} + \vr\vu\otimes \vu: \Grad \mathbf{U}+p(\vr)\Div \mathbf{U}\Big ) \dxdt
\]
\[
+\int_0^{\tau}\int_{\Om} \Big ( \vr \mathbf{U}\cdot \pdt \mathbf{U}+ \vr\vu \cdot \Grad \mathbf{U} \cdot \mathbf{U} \Big )\dxdt
+\int_0^{\tau}\int_{\Om} \mathbb{S} (\Grad \mathbf{U}):
\left(
\Grad \mathbf{U}-\Grad \vu
\right)    \dxdt
\]

\[
+\int_0^{\tau}\int_{\Om}
\left[
\left( 1-\f{\vr}{r} \right)p'(r)\pdt r
-\vr\vu\cdot \f{p'(r)}{r} \Grad r
\right]
\dxdt
+\alpha \int_0^{\tau}\int_{\Om}
\Curl  \mathbf{b} \cdot
\left(
\Curl  \mathbf{b}-\Curl  \vB
\right)
\dxdt
\]

\[
+\int_0^{\tau}\int_{\Om}
\Big(
\pdt \mathbf{b}\cdot
\left(
\mathbf{b}-\vB
\right)
-(\vu \times \vB) \cdot \Curl  \mathbf{b}
\Big)
\dxdt
-\int_0^{\tau}\int_{\Om}
\left(\Curl  \vB \times \vB \right)\cdot \mathbf{U}
\dxdt
\]
\beq\label{lbn8}
-\int_0^{\tau}\int_{\overline{\Om}} \Grad  \mathbf{U} :{\rm d} \mu_c(t)\dt
-\int_0^{\tau}\int_{\overline{\Om}} \mathbf{U} \cdot {\rm d} \mu_B(t)\dt.
\eeq

\subsection{Weak--strong uniqueness principle}\label{west}

The aim of this part is to estimate the right hand side of \eqref{lbn8} towards the proof of weak--strong uniqueness principle. The strategy consists of the following steps:
\begin{itemize}
\item{
Setting the classical solution $(\widetilde{\vr},\widetilde{\vu},\widetilde{\vB})$ as the test function $(r,\mathbf{U},\mathbf{b})$ in the relative energy inequality (\ref{lbn8});
}
\item{
Estimating each term on the right-hand side of the relative energy inequality \eqref{lbn8} in a suitable manner;
}
\item{
Applying Gronwall-type argument to derive the expected results.
}
\end{itemize}
Let $(\widetilde{\vr},\widetilde{\vu},\widetilde{\vB})$ be a classical solution to (\ref{pde})--(\ref{ini_c}) starting from the smooth initial data $(\vr_0,\vu_0,\vB_0)$ with strictly positive $\vr_0$ and $\Div \vB=0$. Let $(\vr,\vu,\vB)$ be a DW solution to (\ref{pde})--(\ref{ini_c}) emanating from $(\vr_0,\vr_0\vu_0,\vB_0)$. It follows from (\ref{lbn8}) that
\[
\mathcal{E}\Big((\vr,\vu,\vB)\,\Big|\,(\widetilde{\vr},\widetilde{\vu},\widetilde{\vB})\Big)(\tau)
+\int_0^{\tau}\int_{\Om}   \mathbb{S}(\Grad\vu-\Grad \widetilde{\vu}):(\Grad \vu-\Grad \widetilde{\vu})
\dxdt
\]
\[
+\alpha \int_0^{\tau}\int_{\Om} |  \Curl  (\vB -\widetilde{\vB})|^2\dxdt
+\int_{\overline{\Om} }{\rm d}\mathfrak{D}(\tau)
+ \int_0^{\tau}
  \int_{\overline{\Om}}  {\rm d}\mathfrak{C}
\]
\[
\leq -\int_0^{\tau}\int_{\Om} \Big ( \vr\vu\cdot \pdt \widetilde{\vu} + \vr\vu\otimes \vu: \Grad \widetilde{\vu}+p(\vr)\Div \widetilde{\vu}\Big ) \dxdt
\]
\[
+\int_0^{\tau}\int_{\Om} \Big ( \vr\widetilde{\vu}\cdot \pdt \widetilde{\vu}+ \vr\vu \cdot \Grad \widetilde{\vu} \cdot \widetilde{\vu} \Big )\dxdt
+\int_0^{\tau}\int_{\Om} \mathbb{S} (\Grad \widetilde{\vu}):
\left(
\Grad \widetilde{\vu}-\Grad \vu
\right)    \dxdt
\]
\[
+\int_0^{\tau}\int_{\Om}
\left[
\left( 1-\f{\vr}{\widetilde{\vr}} \right)p'(\widetilde{\vr})\pdt \widetilde{\vr}
-\vr\vu\cdot \f{p'(\widetilde{\vr})}{\widetilde{\vr}} \Grad \widetilde{\vr}
\right]
\dxdt
+\alpha \int_0^{\tau}\int_{\Om}
\Curl  \widetilde{\vB} \cdot
\left(
\Curl  \widetilde{\vB}-\Curl \vB
\right)
\dxdt
\]
\[
+\int_0^{\tau}\int_{\Om}
\left[
\pdt \widetilde{\vB}\cdot
\left(
\widetilde{\vB}-\vB
\right)
-(\vu \times \vB) \cdot \Curl  \widetilde{\vB}
\right]
\dxdt
-\int_0^{\tau}\int_{\Om}
(\Curl  \vB \times \vB )\cdot \widetilde{\vu}
\dxdt
\]
\beq\label{lbn9}
-\int_0^{\tau}\int_{\overline{\Om}} \Grad  \widetilde{\vu} :{\rm d} \mu_c(t)\dt
-\int_0^{\tau}\int_{\overline{\Om}} \widetilde{\vu} \cdot {\rm d} \mu_B(t)\dt.
\eeq
In light of the compatibility conditions (\ref{lbb8}), it holds that
\[
\left|-\int_0^{\tau}\int_{\overline{\Om}} \Grad  \widetilde{\vu} :{\rm d} \mu_c(t)\dt
-\int_0^{\tau}\int_{\overline{\Om}} \widetilde{\vu} \cdot {\rm d} \mu_B(t)\dt
\right|
\]
\begin{equation*}\label{lbn10}
\lesssim
\epsilon  \int_0^{\tau}
  \int_{\overline{\Om}}  {\rm d}\mathfrak{C}
  + \f{1}{\epsilon}
\int_0^{\tau} \int_{\overline{\Om}}{\rm d}\mathfrak{D}(t)\dt,
\end{equation*}
where $\epsilon>0$ is chosen to be sufficiently small.
Using the hypothesis that $(\widetilde{\vr},\widetilde{\vu},\widetilde{\vB})$ solves (\ref{pde})--(\ref{ini_c}) in the classical sense, i.e.,
\begin{equation}\label{lbn11}
\left\{\begin{aligned}
& \pdt \widetilde{\vr}+\Div (\widetilde{\vr} \,\widetilde{\vu})=0,\\
& \widetilde{\vr} \left(\pdt \widetilde{\vu}+\widetilde{\vu}\cdot \Grad \widetilde{\vu} \right)
+\Grad p(\widetilde{\vr})=\Div \mathbb{S}(\Grad \widetilde{\vu})+\Curl \widetilde{\vB}\times \widetilde{\vB},\\
& \pdt \widetilde{\vB}=\Curl  (\widetilde{\vu} \times \widetilde{\vB})- \alpha\Curl  (\Curl  \widetilde{\vB}),\\
& \Div \widetilde{\vB}=0, \\
\end{aligned}\right.
\end{equation}
we furthermore simplify the right-hand side of (\ref{lbn9}) as follows. Since this process is straightforward and similar to the compressible Navier-Stokes system (see \cite{FGSGW}), the details are omitted.
\[
\mathcal{E}\Big((\vr,\vu,\vB)\,\Big|\,(\widetilde{\vr},\widetilde{\vu},\widetilde{\vB})\Big)(\tau)
+\int_0^{\tau}\int_{\Om} \Big(  \mathbb{S}(\Grad\vu-\Grad \widetilde{\vu}):(\Grad \vu-\Grad \widetilde{\vu})
\Big)\dxdt
\]
\[
+\alpha \int_0^{\tau}\int_{\Om} |  \Curl  (\vB -\widetilde{\vB})|^2\dxdt
+\int_{\overline{\Om} }{\rm d}\mathfrak{D}(\tau)
+ \int_0^{\tau}
  \int_{\overline{\Om}}  {\rm d}\mathfrak{C}
\]
\[
\aleq \int_0^{\tau}\int_{\Om}
\vr(\vu-\widetilde{\vu}) \cdot \Grad \widetilde{\vu}\cdot (\widetilde{\vu}-\vu)
\dxdt
+\int_0^{\tau}\int_{\Om} \mathbb{S} (\Grad \widetilde{\vu}):
\left(
\Grad \widetilde{\vu}-\Grad \vu
\right)    \dxdt
\]
\[
+\int_0^{\tau}\int_{\Om}
\vr(\widetilde{\vu}-\vu)    \cdot
\f{1}{\widetilde{\vr}} \Div \mathbb{S}(\Grad \widetilde{\vu})
\dxdt
-\int_0^{\tau}\int_{\Om}
\Big( p(\vr)-p(\widetilde{\vr}) -p'(\widetilde{\vr}) (\vr-\widetilde{\vr})\Big)  \Div \widetilde{\vu}
\dxdt
\]
\[
+\alpha \int_0^{\tau}\int_{\Om}
\Curl  \widetilde{\vB} \cdot
\left(
\Curl  \widetilde{\vB}-\Curl  \vB
\right)
\dxdt
+\int_0^{\tau}\int_{\Om}
\vr(\widetilde{\vu}-\vu)   \cdot
\f{1}{\widetilde{\vr}} \left(\Curl \widetilde{\vB}\times \widetilde{\vB}\right)
\dxdt
\]
\[
+\int_0^{\tau}\int_{\Om}
\Big(
\pdt \widetilde{\vB}\cdot
(
\widetilde{\vB}-\vB
)
-(\vu \times \vB) \cdot \Curl  \widetilde{\vB}
\Big)
\dxdt
\]
\beq\label{lbn12}
-\int_0^{\tau}\int_{\Om}
(\Curl  \vB \times \vB )\cdot \widetilde{\vu}
\dxdt
+ \int_0^{\tau} \int_{\overline{\Om}}{\rm d}\mathfrak{D}(t)\dt.
\eeq
Notice that the integrals involved with the magnetic field may be rewritten as, using (\ref{lbn11})$_3$,
\[
\alpha \int_0^{\tau}\int_{\Om}
\Curl  \widetilde{\vB} \cdot
\left(
\Curl  \widetilde{\vB}-\Curl  \vB
\right)
\dxdt
+\int_0^{\tau}\int_{\Om}
\vr(\widetilde{\vu}-\vu)   \cdot
\f{1}{\widetilde{\vr}}\left( \Curl \widetilde{\vB}\times \widetilde{\vB}\right)
\dxdt
\]
\[
+\int_0^{\tau}\int_{\Om}
\Big(
\pdt \widetilde{\vB}\cdot
\left(
\widetilde{\vB}-\vB
\right)
-(\vu \times \vB )\cdot \Curl  \widetilde{\vB}
\Big)
\dxdt
-\int_0^{\tau}\int_{\Om}
(\Curl  \vB \times \vB )\cdot \widetilde{\vu}
\dxdt
\]
\[
=\int_0^{\tau}\int_{\Om}
(\vr-\widetilde{\vr})(\widetilde{\vu}-\vu)    \cdot
\f{1}{\widetilde{\vr}}\left( \Curl \widetilde{\vB}\times \widetilde{\vB} \right)
\dxdt
+\int_0^{\tau}\int_{\Om}
(\widetilde{\vu}-\vu)    \cdot
 \left( \Curl \widetilde{\vB}\times \widetilde{\vB} \right)
\dxdt
\]
\[
+\int_0^{\tau}\int_{\Om}
\left[
\Curl  (\widetilde{\vu} \times \widetilde{\vB}) \cdot
\left(
\widetilde{\vB}-\vB
\right)
-(\vu \times \vB)\cdot \Curl  \widetilde{\vB}
\right]
\dxdt
+\int_0^{\tau}\int_{\Om}
\Curl  \vB   \cdot
(\widetilde{\vu}\times \vB )
\dxdt
\]
\[
=\int_0^{\tau}\int_{\Om}
(\vr-\widetilde{\vr})(\widetilde{\vu}-\vu)    \cdot
\f{1}{\widetilde{\vr}}\left( \Curl \widetilde{\vB}\times \widetilde{\vB} \right)
\dxdt
+\int_0^{\tau}\int_{\Om}
   \Curl  (\vB -\widetilde{\vB})
   \cdot
   \Big( \widetilde{\vu} \times (\vB-\widetilde{\vB})
   \Big )
   \dxdt
\]
\beq\label{lbn13}
 +\int_0^{\tau}\int_{\Om}
  \Curl  \widetilde{\vB} \cdot
 \Big( (\vu-\widetilde{\vu})\times(\widetilde{\vB}-\vB)\Big)
  \dxdt.
\eeq
Moreover, it holds that
\[
\left|
\int_0^{\tau}\int_{\Om}
   \Curl  (\vB -\widetilde{\vB})
   \cdot
   \Big( \widetilde{\vu} \times (\vB-\widetilde{\vB})
   \Big )
   \dxdt
\right|
\]
\beq\label{lbn14}
\aleq \epsilon \int_0^{\tau}\int_{\Om} \left|  \Curl  (\vB -\widetilde{\vB})\right|^2\dxdt
+c(\epsilon) \int_0^{\tau}\int_{\Om}
|\vB-\widetilde{\vB}|^2
\dxdt;
\eeq
\[
\left|
\int_0^{\tau}\int_{\Om}
  \Curl  \widetilde{\vB} \cdot
 \Big( (\vu-\widetilde{\vu})\times(\widetilde{\vB}-\vB)\Big)
  \dxdt
\right|
\]
\beq\label{lbn15}
\aleq \epsilon \int_0^{\tau}\int_{\Om}  |\vu-\widetilde{\vu}|^2 \dxdt
+c(\epsilon) \int_0^{\tau}\int_{\Om}
|\vB-\widetilde{\vB}|^2
\dxdt.
\eeq
Due to the generalized Korn-type inequality,
\beq\label{lbn16}
\int_0^{\tau}\int_{\Om} |\vu-\widetilde{\vu}|^2\dxdt
\lesssim
\int_0^{\tau}\int_{\Om}  \mathbb{S}(\Grad\vu-\Grad \widetilde{\vu}):(\Grad \vu-\Grad \widetilde{\vu})
\dxdt.
\eeq
In addition, the isentropic law of pressure function yields
\[
\left|
\int_0^{\tau}\int_{\Om}
\Big( p(\vr)-p(\widetilde{\vr}) -p'(\widetilde{\vr}) (\vr-\widetilde{\vr})\Big)  \Div \widetilde{\vu}
\dxdt
\right|
\]
\beq\label{lbn17}
\aleq
\int_0^{\tau}\int_{\Om}
\Big(\mathcal{H}(\vr)-\mathcal{H}(\widetilde{\vr}) -\mathcal{H}'(\widetilde{\vr}) (\vr-\widetilde{\vr})\Big)
\dxdt.
\eeq

Consequently, combining (\ref{lbn12})--(\ref{lbn17}) and choosing $\epsilon>0$ suitably small give rise to
\[
\mathcal{E}\Big((\vr,\vu,\vB)\,\Big|\,(\widetilde{\vr},\widetilde{\vu},\widetilde{\vB})\Big)(\tau)
+\int_0^{\tau}\int_{\Om}   \mathbb{S}(\Grad\vu-\Grad \widetilde{\vu}):(\Grad \vu-\Grad \widetilde{\vu})
\dxdt
\]
\[
+\alpha \int_0^{\tau}\int_{\Om} |  \Curl  (\vB -\widetilde{\vB})|^2\dxdt
+\int_{\overline{\Om} }{\rm d}\mathfrak{D}(\tau)
+ \int_0^{\tau}
  \int_{\overline{\Om}}  {\rm d}\mathfrak{C}
\]
\[
\aleq
\int_0^{\tau}\int_{\Om} \mathbb{S} (\Grad \widetilde{\vu}):
\left(
\Grad \widetilde{\vu}-\Grad \vu
\right)    \dxdt
+\int_0^{\tau}\int_{\Om}
\vr(\widetilde{\vu}-\vu)   \cdot
\f{1}{\widetilde{\vr}} \,\Div \mathbb{S}(\Grad \widetilde{\vu})
\dxdt
\]
\[
+\int_0^{\tau}\int_{\Om}
(\vr-\widetilde{\vr})(\widetilde{\vu}-\vu)    \cdot
\f{1}{\widetilde{\vr}}\left( \Curl \widetilde{\vB}\times \widetilde{\vB} \right)
\dxdt
\]
\beq\label{lbn18}
+
\int_0^{\tau}
\mathcal{E}\Big((\vr,\vu,\vB)\,\Big|\,(\widetilde{\vr},\widetilde{\vu},\widetilde{\vB})\Big)(t)
\dt
+\int_0^{\tau} \int_{\overline{\Om}}{\rm d}\mathfrak{D}(t)\dt.
\eeq
Following \cite{FGSGW,FNS}, we estimate the remaining integrals as follows. Let $\chi$ be a cut-off function such that
\begin{equation*}\label{lbn19}
\left\{\begin{aligned}
& \chi \in C _c^{\infty}((0,\infty)),\\
& 0\leq \chi \leq 1,\\
& \chi(\vr)=1 \text{ if }\vr \in [\inf{\widetilde{\vr}},\sup{\widetilde{\vr}}].\\
\end{aligned}\right.
\end{equation*}
Then we may write
\[
\left|
\int_0^{\tau}\int_{\Om}
(\vr-\widetilde{\vr})(\widetilde{\vu}-\vu)   \cdot
\f{1}{\widetilde{\vr}}\left( \Curl \widetilde{\vB}\times \widetilde{\vB} \right)
\dxdt
\right|
\]
\beq\label{lbn20}
\aleq
\int_0^{\tau}\int_{\Om}
\chi(\vr)|\vr-\widetilde{\vr}||\widetilde{\vu}-\vu|
\dxdt
+
\int_0^{\tau}\int_{\Om}
(1-\chi(\vr))|\vr-\widetilde{\vr}||\widetilde{\vu}-\vu|
\dxdt.
\eeq
The first integral on the right-hand side of (\ref{lbn20}) is bounded through
\[
\int_0^{\tau}\int_{\Om}
\chi(\vr)|\vr-\widetilde{\vr}||\widetilde{\vu}-\vu|
\dxdt
\]
\[
\aleq
\int_0^{\tau}\int_{\Om}
\f{1}{2}\vr|\widetilde{\vu}-\vu|^2
\dxdt
+
\int_0^{\tau}\int_{\Om}
\f{1}{2}\f{\chi^2(\vr)}{\vr} |\vr-\widetilde{\vr}|^2
\dxdt
\]
\beq\label{lbn21}
\aleq
\int_0^{\tau}
\mathcal{E}\Big((\vr,\vu,\vB)\,\Big|\,(\widetilde{\vr},\widetilde{\vu},\widetilde{\vB})\Big)(t)
\dt.
\eeq
To estimate the second integral on the right-hand side of (\ref{lbn20}), we make a further decomposition, i.e.,
\[
1-\chi(\vr)=\chi_1(\vr)+\chi_2(\vr)
\]
such that
\[
\text{supp}\chi_1 \subset [0,\inf{\widetilde{\vr}}],\,\,
\text{supp}\chi_2 \subset [\sup{\widetilde{\vr}},\infty].
\]
It follows from (\ref{lbn16}) that
\[
\int_0^{\tau}\int_{\Om}
\chi_1(\vr)|\vr-\widetilde{\vr}||\widetilde{\vu}-\vu|
\dxdt
\]
\[
\aleq
\epsilon
\int_0^{\tau}\int_{\Om}  |\vu-\widetilde{\vu}|^2   \dxdt
+c(\epsilon)
\int_0^{\tau}\int_{\Om}
\chi_1^2(\vr)|\vr-\widetilde{\vr}|^2
\dxdt
\]
\beq\label{lbn22}
\aleq
 \epsilon\int_0^{\tau}\int_{\Om} \Big(  \mathbb{S}(\Grad\vu-\Grad \widetilde{\vu}):(\Grad \vu-\Grad \widetilde{\vu})
\Big)\dxdt
+c(\epsilon)
\int_0^{\tau}
\mathcal{E}\Big((\vr,\vu,\vB)\,\Big|\,(\widetilde{\vr},\widetilde{\vu},\widetilde{\vB})\Big)(t)
\dt.
\eeq
Clearly,
\[
\int_0^{\tau}\int_{\Om}
\chi_2(\vr)|\vr-\widetilde{\vr}||\widetilde{\vu}-\vu|
\dxdt
\]
\[
\aleq
\int_0^{\tau}\int_{\Om}
\chi_2(\vr)\vr|\widetilde{\vu}-\vu| ^2
\dxdt
+\int_0^{\tau}\int_{\Om}
\chi_2(\vr)\vr
\dxdt
\]
\beq\label{lbn23}
\aleq
\int_0^{\tau}
\mathcal{E}\Big((\vr,\vu,\vB)\,\Big|\,(\widetilde{\vr},\widetilde{\vu},\widetilde{\vB})\Big)(t)
\dt.
\eeq
Taking (\ref{lbn20})--(\ref{lbn23}) into account, we see
\[
\left|
\int_0^{\tau}\int_{\Om}
(\vr-\widetilde{\vr})(\widetilde{\vu}-\vu)    \cdot
\f{1}{\widetilde{\vr}}\left( \Curl \widetilde{\vB}\times \widetilde{\vB} \right)
\dxdt
\right|
\]
\beq\label{lbn24}
\aleq
 \epsilon\int_0^{\tau}\int_{\Om} \Big(  \mathbb{S}(\Grad\vu-\Grad \widetilde{\vu}):(\Grad \vu-\Grad \widetilde{\vu})
\Big)\dxdt
+
c(\epsilon)
\int_0^{\tau}
\mathcal{E}\Big((\vr,\vu,\vB)\,\Big|\,(\widetilde{\vr},\widetilde{\vu},\widetilde{\vB})\Big)(t)
\dt.
\eeq
Finally, notice also that the first two integrals on the right-hand side of (\ref{lbn18}) are estimated as above upon observing that
\[
\int_0^{\tau}\int_{\Om} \mathbb{S} (\Grad \widetilde{\vu}):
\left(
\Grad \widetilde{\vu}-\Grad \vu
\right)    \dxdt
+\int_0^{\tau}\int_{\Om}
\vr(\widetilde{\vu}-\vu)    \cdot
\f{1}{\widetilde{\vr}} \Div \mathbb{S}(\Grad \widetilde{\vu})
\dxdt
\]
\beq\label{lbn25}
=
\int_0^{\tau}\int_{\Om}
(\vr-\widetilde{\vr})(\widetilde{\vu}-\vu)   \cdot
\f{1}{\widetilde{\vr}}
\Div \mathbb{S}(\Grad \widetilde{\vu})
\dxdt.
\eeq
Combining (\ref{lbn18}), (\ref{lbn24})--(\ref{lbn25}) and fixing $\epsilon>0$ sufficiently small shows that
\[
\mathcal{E}\Big((\vr,\vu,\vB)\,\Big|\,(\widetilde{\vr},\widetilde{\vu},\widetilde{\vB})\Big)(\tau)
+\int_0^{\tau}\int_{\Om} \Big(  \mathbb{S}(\Grad\vu-\Grad \widetilde{\vu}):(\Grad \vu-\Grad \widetilde{\vu})
\Big)\dxdt
\]
\[
+\alpha \int_0^{\tau}\int_{\Om} |  \Curl  (\vB -\widetilde{\vB})|^2\dxdt
+\int_{\overline{\Om} }{\rm d}\mathfrak{D}(\tau)
+ \int_0^{\tau}
  \int_{\overline{\Om}}  {\rm d}\mathfrak{C}
\]
\begin{equation*}\label{lbn26}
\aleq
\int_0^{\tau}
\mathcal{E}\Big((\vr,\vu,\vB)\,\Big|\,(\widetilde{\vr},\widetilde{\vu},\widetilde{\vB})\Big)(t)
\dt+ \int_0^{\tau} \int_{\overline{\Om}}{\rm d}\mathfrak{D}(t)\dt.
\end{equation*}
As a direct application of Gronwall's inequality, we conclude that
\[
\vr=\tilde{\vr},\,\,\,\, \vu=\tilde{\vu},\,\,\,\,\vB=\tilde{\vB},
\text{    in   } (0,T)\times \Om,
\]
\[
\mu_c=\mathbf{0},\,\,\,\, \mu_B=\mathbf{0},\,\,\,\, \mathfrak{D}=0,\,\,\,\, \mathfrak{C}=0,
\]
thus completely finishing the proof of Theorem \ref{Th1}.
\qed

\section{Convergence}\label{sec_con}
In this section we prove another main result, that is the convergence of a consistent approximation stated in Theorem~\ref{Th2}. As an application of this theorem, we will also show the convergence analysis of two mixed finite volume-finite element methods.
\subsection{Convergence of a consistent approximation}
In this subsection we prove Theorem~\ref{Th2} for the convergence of a consistent approximation $(\vrh, \vuh, \vBh)$ in two steps, that are the convergences towards a DW solution and towards a classical solution.
\subsubsection{Convergence to a DW solution}
\begin{proof}[Proof of Item 1 of Theorem~\ref{Th2}]
As $(\vrh, \vuh, \vBh)$ is a consistent approximation in the sense of Definition~\ref{def_ca}, it satisfies the stability property~\eqref{es}. Consequently, we derive for suitable subsequences, not relabelled, that
\[
\begin{split}
\vrh &\to \vr \ \mbox{weakly-(*) in}\ L^\infty(0,T; L^\gamma(\Omega)),\ \vr \geq 0,
\\
\vuh, \Pi_h \vuh &\to \vu \ \mbox{weakly in}\  L^2((0,T)\times \Omega; \R^d),  \mbox{ where } \vu \in L^2(0,T; W^{1,2}(\Omega;\R^d)),\\
  & \mbox{ and in the case of Dirichlet boundary conditions } \vu \in L^2(0,T; W_0^{1,2}(\Omega;\R^d)), \\
 \vrh  \Pi_h \vuh  &\to \overline{\vr \vu} \ \  \mbox{weakly-(*) in}\ L^\infty(0,T; L^{\frac{2\gamma}{\gamma + 1}}(\Omega; \R^d)),
\\
\vBh &\to \vB \ \mbox{weakly-(*) in}\ L^\infty(0,T; L^2(\Omega;\R^d)),
\\
\Curlh \vBh &\to \Curl \vB \ \mbox{weakly in}\ L^2((0,T)\times \Omega; \R^d),
\\
\vuh\times \vBh &\to \overline{\vu \times \vB} \ \mbox{weakly in}\ L^2(0,T; L^{\f32}(\Omega;\R^d)),
\\
\vrh \Pi_h \vuh \otimes \vuh + p(\vrh) \mathbb{I} &
\to  \overline{      1_{\vr>0}\f{\mathbf{m}\otimes \mathbf{m}}{\vr} +p(\vr)\mathbb{I}    }
 \ \mbox{weakly-(*) in} \   L^\infty(0,T; \mathcal{M}(\overline{\Omega};\R^{d\times d}_{sym})) ,
\\
\Curlh \vBh \times \vBh & \to \overline{ \Curl \mathbf{B}\times \mathbf{B}}
 \ \mbox{weakly-(*) in} \  L^2(0,T; \mathcal{M}(\overline{\Omega};\R^d)) ,
 \\
 \S(\Gradh \vuh) : \Gradh \vuh  &\to \overline{\mathbb{S}(\Grad \vu):\Grad \vu} \mbox{ in }\mathcal{M}^{+}([0,T]\times \overline{\Om}),
 \\
 |\Curlh \vBh|^2 &\to \overline{ |\Curl \vB|^2 }\mbox{ in }\mathcal{M}^{+}([0,T]\times \overline{\Om}),
 \\
 \frac{1}{2} \vrh \abs{\Pi_h  \vuh}^2  + \Hc(\vrh) +\frac{1}{2} \abs{\vBh}^2  &\to \overline{  \f{1}{2}\vr |\vu|^2+\mathcal{H}(\vr)+\f{1}{2}|\mathbf{B}|^2 }\mbox{  weakly-(*) in }L^{\infty}(0,T; \mathcal{M}^{+}(\overline{\Om})).
\end{split}
\]

With the uniform bounds at hand, we may invoke Lemma 3.7 in Abbatiello et al. \cite{AnFeNo} to deduce that
\[
\overline{\vr \vu}=\vr\vu,\,\,\,\, \overline{\vu \times \vB}=\vu\times \vB.
\]
We then set
\[
\mu_c:=\overline{      1_{\vr>0}\f{\mathbf{m}\otimes \mathbf{m}}{\vr} +p(\vr)\mathbb{I}    }
-\left(
 1_{\vr>0}\f{\mathbf{m}\otimes \mathbf{m}}{\vr} +p(\vr)\mathbb{I}
\right),
\]
\[
\mu_B:=\overline{ \Curl \mathbf{B}\times \mathbf{B}}
-( \Curl \mathbf{B}\times \mathbf{B} ),
\]
\[
\mathfrak{D}:=\overline{  \f{1}{2}\vr |\vu|^2+\mathcal{H}(\vr)+\f{1}{2}|\mathbf{B}|^2 }
-\left(\f{1}{2}\vr |\vu|^2+\mathcal{H}(\vr)+\f{1}{2}|\mathbf{B}|^2\right),
\]
\[
\mathfrak{C}=\Big(\overline{\mathbb{S}(\Grad \vu):\Grad \vu}
-\mathbb{S}(\Grad \vu):\Grad \vu \Big)+\alpha \Big(\overline{ |\Curl \vB|^2 }
-|\Curl \vB|^2 \Big),
\]

Knowing the above limit, we are ready to pass to the limit $h \to 0$ in the consistency formulation \eqref{cP1}--\eqref{cP4} and the energy stability \eqref{es}.  We get the following formulae for the limit functions
\begin{equation*} \label{M1}
\left[ \intO{ \vr  \phi } \right]_{t=0}^{t=\tau}  =
\int_0^\tau \intO{ \Big( \vr  \pdt \phi +  \vr \vu  \cdot \Grad \phi \Big)} \dt  ,
\end{equation*}
for any $\phi \in C_c^M([0,T] \times \Ov{\Omega})$ and for some $M \geq 1$;
\begin{equation*} \label{M2}
\begin{split}
& \left[ \intO{ \vr \vu \cdot \vv } \right]_{t=0}^{t=\tau}  =
\int_0^{\tau} \intO{ \Big(   \vr \vu \cdot \partial_t \vv +  \vr \vu \otimes \vu +p(\vr) \I   : \Grad \vv   \Big)} \dt
\\&
 -   \int_0^{\tau} \intO{  \S( \Grad \vu) : \Grad \vv}  \dt
+ \int_0^{\tau} \intO{ ( \Curl \vB \times \vB)  \cdot  \vv}\dt
\\ &
+
 \int_0^{\tau}\int_{\overline{\Om}} \Grad  \vv :{\rm d} \mu_c(t)\dt
 +
 \int_0^{\tau}\int_{\overline{\Om}} \vv \cdot {\rm d} \mu_B(t)\dt
\end{split}
\end{equation*}
for any $\vv \in C_c^M([0,T] \times \Omega; \R^d)$ and for some $M \geq 1$;
\begin{equation*} \label{M3}
 \left[ \intO{   \vB \cdot \vC }\right]_{t=0}^{t=\tau}
= \int_0^{\tau}  \int_{\Om} \Big(\vB \cdot \pdt \vC
- \alpha \Curl  \vB  \cdot \Curl \vC
+  (\vu\times \vB) \cdot \Curl\vC\Big)
\dxdt
\end{equation*}
 for any $\vC \in C_c^M([0,T]\times \overline{\Om};\R^d),\,\vv\times \vn|_{\p \Om}=\mathbf{0}$ and for some $M \geq 1$;
\begin{equation*}\label{M4}
\intO{   \vB  \cdot \Grad \psi} = 0
\end{equation*}
for any $\psi \in C^M( \overline{\Om}) \cap L_0^2(\Om) $ and for some $M \geq 1$;
\[
\int_{\Om} \left[\f{1}{2}\vr |\vu|^2+\f{1}{2}|\mathbf{B}|^2+\mathcal{H}(\vr)\right](\tau,x)\dx
+\int_0^{\tau}\int_{\Om} \left (  \mathbb{S}(\Grad\vu):\Grad \vu
+\alpha \left|  \Curl  \mathbf{B} \right|^2
\right )\dxdt
\]
\begin{equation*} \label{M5}
+\int_{\overline{\Om} }{\rm d}\mathfrak{D}(\tau)
+ \int_0^{\tau}
  \int_{\overline{\Om}}  {\rm d}\mathfrak{C}
\leq
\int_{\Om}\left[ \f{1}{2}\f{|\mathbf{m}_0|^2}{\vr_0}+\f{1}{2}|\mathbf{B}_0|^2+\mathcal{H}(\vr_0)\right]\dx.
\end{equation*}
for a.e. $\tau \in (0,T)$;

Finally, the compatibility conditions (\ref{lbb8}) and (\ref{lbb8-2}) hold owing to \cite[Lemma 2.1]{FGSGW} and the crucial observations:
\[
\Big|\vr \vu \otimes \vu +p(\vr)\mathbb{I}\Big|
\lesssim \f{1}{2}\vr |\vu|^2+\mathcal{H}(\vr)
\lesssim \f{1}{2}\vr |\vu|^2+\mathcal{H}(\vr)+\f{1}{2}|\mathbf{B}|^2,
\]
\[
\big| \Curl \vB \times \vB \big|
\lesssim
\epsilon |\Curl \vB|^2 +\f{1}{\epsilon}|\vB|^2
\lesssim
\epsilon \big( |\Curl \vB|^2 +\mathbb{S}(\Grad\vu):\Grad \vu \big)
\]
\[
+\f{1}{\epsilon}\left[\f{1}{2}\vr |\vu|^2+\f{1}{2}|\mathbf{B}|^2+\mathcal{H}(\vr)\right], \,\,\text{   for any  }\epsilon>0.
\]

Consequently, collecting the above identities and relations, we conclude that the weak limit $(\vr,\vu,\vB)$, with the associated $\mu_c,\mathfrak{D},\mathfrak{C}$, generated by the consistent approximation $(\vrh, \vuh, \vBh)$ for $h \to 0$, represents a DW solution of the MHD system \eqref{pde} in the sense of Definition \ref{def_dw}. This proves Item 1 of Theorem~\ref{Th2}.
\end{proof}

\subsubsection{Convergence to the classical solution}
\begin{proof}[Proof of Item 2 of Theorem~\ref{Th2}]
Combining Item 1 of Theorem~\ref{Th2} and Theorem~\ref{Th1} we immediately obtain Item 2 of Theorem~\ref{Th2}; that is the convergence of the consistent approximation $(\vrh, \vuh, \vBh)$ towards the classical solution.
\end{proof}
\medskip

We have built a general theory in Theorem~\ref{Th2} that a consistent approximation converges to the DW solution as well as the classical solution (on its lifespan). Next, we show the application of this theory in the convergence analysis of numerical solutions by two examples. The only gap here is whether the numerical solution is a consistent approximation in the sense of Definition~\ref{def_ca}.
%
\subsection{Example-I}
In the first example, we propose a mixed FV--FE approximation adapted from the Navier-Stokes solver of Karper~\cite{Karper} and the magnetic solver of Ding and Mao~\cite{Ding} with no-slip boundary conditions. 
To begin, we introduce the  necessary notations. 
\paragraph{Mesh.}  Let $\Oh$ be a regular and quasi-uniform triangulation of the bounded domain $\Omega$, $\mcE$ be the set of all $(d-1)$-dimensional faces, $\mcEe=\mcE \cap \pd \Omega$ be the exterior faces, $\mcEi=\mcE \setminus \mcEe$ be the interior faces, and $\facesK$ be the set of all faces of an arbitrary element $K$.
We denote $\sigma=K|L\in \mcEi$ as the common face of two neighbouring elements $K$ and $L$.
Further, we denote $\bfn_\sigma$ as the outer normal of a face $\sigma \in \faces$ and $\bfn_{\sigma,K}$ as the unit normal vector pointing outwards $K$ if $\sigma \in \facesK$.
The size of the mesh (maximal diameter of all elements) is supposed to be a positive parameter $h<1$.
\paragraph{Function spaces.}\label{sec:spaces}
We define on $\grid$ discrete function spaces $Q_h, V_h, \Nh$, that are spaces of piecewise constants, piecewise linear Crouzeix--Ravairt elements,  and lowest order $\Hcurl$-N\'{e}d\'{e}lec edge elements, respectively.
\begin{equation*}
 Q_h \equiv \left\{ v \in L^2(\Omega) \middle| \; v \vert_K \in \mcP^1_0(K) \ K \in \Oh \right\},
\end{equation*}
\begin{equation*}
 V_h \equiv \left\{ \vv \in L^2(\Omega) \middle| \; \vv \vert_K \in \mcP^d_1(K)  \forall\,  K \in \Oh;\intG{ \jump{\vv}} =0 \; \forall\,  \sigma \in \mcEi \right\},
\end{equation*}
\begin{equation*}
 \Nh \equiv \left\{ \vv, \Curl \vv \in L^2(\Omega) \middle| \;  \vv \vert_K \in \mcP_0^d\oplus\mcP_0^1 \bfx \; \forall\;  K \in \Oh;
 \intG{ \jump{\vv \times \bfn}} =0 \; \forall\;  \sigma \in \mcEi
 \right\}, 
\end{equation*}
where $\mcP^d_n(K)$ (resp. $\mcP^1_n(K)$) denotes the space of polynomials of degree not greater than $n$ on element $K$ for $d$-dimensional vector valued functions (resp. for scalar functions).
The spaces $Q_h, V_h$ and $\Nh$ shall be used for the approximation of the discrete density, velocity and magnetic field, respectively. 
In addition, for the purpose of designing weakly divergence-free magnetic field (see Lemma~\ref{lem_b}), we introduce the following space
 \[W_h = \left\{ v \in W^{1,2} \cap L^2_0 \middle| \; v\vert_K \in \mcP_1^1(K) \; \forall \; K \in \Oh \right\}. 
\]
 It should be mentioned that for any $\psi_h \in W_h$ we have $\nabla \phi_h \in \Nh$. 

The interpolation operators associated to the function spaces $V_h$, $\Nh$ and $W_h$ are given by
\begin{equation*}\label{proj1}
\PiV :\ W^{1,2}(\Omega) \rightarrow V_h, \quad
\PiN :\ W^{1,2}(\Omega) \rightarrow \Nh, \quad
\PiW :\ L^2_0(\Omega) \rightarrow W_h
\end{equation*}
satisfy the following interpolation estimates 
\begin{equation}\label{IE}
\norm{\vv - \PiVv }_{L^p}   \aleq h \norm{\vv}_{C^1},
\;
\norm{\vv - \PiW \vv }_{L^p}   \aleq h \norm{\vv}_{C^1},
\;
\norm{\vu - \PiN \vu }_{L^p} +\norm{\Curl (\vu -  \PiN \vu) }_{L^p}     \aleq h \norm{\vu}_{C^2}.
\end{equation}
 for any $\vv \in C^1(\Omega)$, $\vu \in C^2(\Omega)$, $p\in[1,\infty]$, see e.g. the monograph of Brezzi et al.~\cite{Brezzi}. 

For simplicity of notation, we denote $\FS=Q_h \times \CR \times \ND$, where
\[ \CR = \left\{ \vv \in V_h \middle|  \intG {\vv} =0\ \forall\,  \sigma \in \mcEe \right\} \mbox{ and }
\ND=\left\{ \vv \in \Nh \middle|    \intG {\vv\times \bfn} =0 \; \forall\;  \sigma \in \mcEe  \right\}.
\]

Next, we suppose the time step size $\TS \approx h$ and denote  $t^k= k\TS$ for $k=1,\ldots,N_T(=T/\TS)$.
For a generic discrete function $v_h$ at time $t^k= k\TS$ we denote it by $v_h^k$ and write $v_h \in L_{\TS}(0,T;Y)$ if $v_h^k \in Y$ for all $k=1,\dots, N_T$ with $Y \in \{Q_h, \CR, \ND, \FS \}$ and 
\[ v_h(t,\cdot) =v_h^0 \mbox{ for } t \leq 0,\ v_h(t,\cdot)=v_h^k \mbox{ for } t\in ((k-1)\TS,k\TS],\ k=1,2,\ldots,N_T.
\]

\paragraph{Some discrete operators.}
We define the discrete time derivative by the backward Euler method
\[
 D_t v_h = \frac{v_h (t) - v_h^{\triangleleft}(t)}{\TS}  \mbox{ for } t\in(0,T) \quad \mbox{ with } \quad   v_h^{\triangleleft}(t) =  v_h(t - \Delta t) .
\]
For any piecewise continuous function $f$, we define its trace on a generic edge 
as
\begin{equation*}
\begin{aligned}
f^{\rm in}|_{\sigma} = \lim_{\delta \rightarrow 0^+} f(\xx -\delta \nG ), \ \forall\  \sigma \in \mcE,
\qquad
f^{\rm out}|_{\sigma} = \lim_{\delta \rightarrow 0^+} f(\xx +\delta \nG ), \ \forall\  \sigma \in \mcEi.
\end{aligned}
\end{equation*}
Note that $f^{\rm out}|_{\mcEe}$ is determined by the boundary condition.
Further, we define the jump and average operators at an edge $\sigma \in \mcE$ as
\begin{equation}\label{op_diff}
\jump{f}_{\sigma} = f^{\rm out} - f^{\rm in} \mbox{ and }
\avg{f} = \frac{f^{\rm out} + f^{\rm in}}{2} ,
\end{equation}
respectively, and denote $\Pi_h \equiv \PiQ$ as the element-wise constant projection, where
\begin{equation*}\label{op_avgk}
\PiQ: L^1(\Omega) \to Q_h. \quad   \left. \PiQ f\right|_K =  \left. \avc{f}\right|_K \equiv \frac{1}{\abs{K}} \intK f   , \ \forall\ K \in \Oh.
\end{equation*}
Next,  we introduce the upwind flux for any function $r_h \in Q_h$ at a generic face  $\sigma \in \facesint$
\begin{align*}\label{Up}
\Up [r_h, \vuh]   =r_h^{\rm up}\us
=r_h^{\rm in} [\us]^+ + r_h^{\rm out} [\us]^-,
\end{align*}
where $\vuh \in V_h$ is the velocity field and 
\begin{equation*}
\us = \frac{1}{|\sigma|} \intG{\vuh} \cdot \vc{n}_\sigma, \quad
[f]^{\pm} = \frac{f \pm |f| }{2} \quad \mbox{and} \quad
r_h^{\rm up} =
\begin{cases}
 r_h^{\rm in} & \mbox{if } \ \us \geq 0, \\
r_h^{\rm out} & \mbox{if } \ \us < 0.
\end{cases}
\end{equation*}
Furthermore, we consider a diffusive numerical flux function of the following form for $\eps>0$
\begin{equation}\label{num_flux}
\begin{aligned}
\Fup(r_h,\vuh)
=\Up[r_h, \vuh] -  h^\eps \jump{ r_h }
.
\end{aligned}
 \end{equation}
It is easy to check for any $\vrh \in Q_h$ and $\vuh \in V_h$ that
\begin{equation}\label{fluxes}
\begin{aligned}
& \sum_{ \sigma \in \facesint } \intG{ \left(\Fup(\vrh \avc{\vuh},\vuh) \jump{\avc{\vuh}} -  \Fup(\vrh ,\vuh) \jump{ \frac{|\avc{\vuh}|^2}{2}}  \right)}
\\ & = - \sum_{ \sigma \in \facesint } \intG{ \left(\frac{1}2 \vrh^{\up} \abs{\us}
+h^\eps \avg{\vrh}   \right)\abs{\jump{\avc{\vuh}}}^2 }  .
\end{aligned}
\end{equation}

For simplicity, we denote
$\co{a}{b} =\left[\min(a,b), \max(a,b) \right]$
and write
$ a \aleq b \mbox{ if } a \le cb$ if $c$ is a positive constant that is independent of the mesh size and time step used in the scheme.
We shall frequently use the abbreviation $\norm{\cdot}_{L^p}$ and $\norm{\cdot}_{L^pL^q}$  for $\norm{\cdot}_{L^p(\Omega)}$ and $\norm{\cdot}_{L^p\left(0,T;L^q(\Omega)\right)}$, respectively.

\paragraph{The numerical method.}
Using the above notations we propose a mixed FV--FE method for the approximation of the MHD system \eqref{pde} named as {\bf Scheme-I}.
\begin{tcolorbox}
{\bf Scheme-I.}
Given the initial values  \eqref{ini_c} we set $(\vrh^0,\vuh^0, \vBh^0) =(\PiQ \vr_0, \PiV \vu_0, \PiN \vB_0)$ and  seek $(\vrh,\vuh, \vBh) \in L_{\TS}(0, T;\FS) $ such that
\begin{subequations}\label{scheme}
\begin{equation}\label{schemeD}
\intO{ D_t \vrh  \phi_h } - \sum_{ \sigma \in \facesint } \intG{  \Fup(\vrh  ,\vuh  )
\jump{\phi_h}   } = 0 \quad \mbox{for all } \phi_h \in Q_h;
\end{equation}
\begin{equation}\label{schemeM}
\begin{aligned}
  \intO{ D_t (\vrh  \auh) \cdot  \vh   } -
  \sum_{ \sigma \in \facesint } \intG{  \Fup(\vrh   \auh,\vuh  ) \cdot
\jump{\avh}   }
+ \mu \intO{  \Gradh \vuh  : \Gradh \vh  } \\ +  \intO{( \nu \Divh \vuh  -p_h ) \Divh \vh  }
 - \intO{ (\Curlh \vBh   \times  \vBh^\triangleleft ) \cdot \vh}
 =0  \quad \mbox{for all } \vh \in \CR ;
\end{aligned}
\end{equation}
\begin{equation}\label{schemeB}
\intOB{ D_t \vBh  \cdot  \vCh
+ \alpha  \Curlh  \vBh  \cdot \Curlh \vCh
-  (\vuh  \times \vBh^\triangleleft) \cdot   \Curlh \vCh
 } \quad \mbox{for all } \vCh \in \ND .
\end{equation}
\end{subequations}
where $\nu = \frac{d-2}{d}\mu+\lambda$,
the discrete operators $\Divh, \Gradh$ and $\Curlh$ are the same as the continuous case on each element. Moreover, the artificial diffusion parameter $\eps$ follows
\begin{equation*}
\eps >0 \mbox{ if } \gamma \geq 2 \quad \mbox{ and }\quad   \eps \in(0, 2 \gamma-1 -d/3) \mbox{ if } \gamma \in(4d /(1+3d),2).
\end{equation*}
\end{tcolorbox}
 {\bf Scheme-I} enjoys the following properties:
\begin{Lemma}[Existence, mass conservation, renormalized continuity, positivity, divergence free]\label{lem_b}
\hfill
\begin{enumerate}
\item {\bf Existence of a numerical solution.}
There exists at least one solution to  {\bf Scheme-I}.
\item {\bf Mass conservation.}
The numerical method~\eqref{scheme} preserves the total mass.
\[
\intO{ \vrh (t) } = \intO{ \vrh (0) } = \intO{ \vr_0 } , \quad \forall\;  t \in[0,T].
\]
\item {\bf Renormalized continuity equation.}
Let $(\vrh,\vuh)\in  Q_h \times \CR$ satisfy the discrete continuity equation \eqref{schemeD} and $b=b(\vr)\in C^2(0,\infty)$. Then the discrete continuity equation \eqref{schemeD} can be renormalized in the sense that
\begin{equation}\label{r1}
\begin{aligned}
&\intO{\left(D_t b(\vrh ) -\big(\vrh  b'(\vrh )-b(\vrh )\big) \Divh \vuh  \right) }
\\&=
- \frac{\TS}2 \intO{ b''(\xi)|D_t \vrh |^2  }
- \sum_{ \sigma \in \facesint } \intG{ b''(\zeta) \jump{  \vrh  } ^2 \left(h^\eps  +  \frac12 | \us | \right) }.
\end{aligned}
\end{equation}
where $\xi \in \co{\vrh^{\triangleleft}}{\vrh }$ and  $\zeta \in \co{\vrh^{\rm in}}{\vrh^{\rm out}}$.
\item {\bf Positivity of the density.} Let $\vr_0>0$. Then any solution to the discrete problem \eqref{scheme} satisfies $\vrh(t)>0$ for $t\in(0,T)$.
\item {\bf Weakly divergence free of magnetic field.}
Let $\Div \vB_0 =0$. Then {\bf Scheme-I} preserves divergence free of magnetic field weakly, meaning that
$\intO{  \vBh \cdot \Gradh \psi_h} =0$ for any $\psi_h \in W_h$.
\end{enumerate}
\end{Lemma}
\begin{proof}
\begin{itemize}
\item
The existence of a numerical solution to \eqref{scheme} can be proven exactly in the same way as in \cite{Ding} via the theorem of topological degree, see also similar result in \cite[Lemma 11.3]{FeLMMiSh}.
\item
Taking $\phi_h \equiv 1$ in the equation of continuity \eqref{schemeD} immediately yields the mass conservation.
\item
We refer to~\cite[Lemma 4.1]{Karper} for the proof of renormalized continuity equation.
\item
Concerning the positivity of density, we refer to \cite[Lemma 8.3]{FeLMMiSh} for the proof.
\item We refer to \cite[Remark 6]{Ding} for the proof of weakly divergence-free of magnetic field.
\end{itemize}
\end{proof}

\subsubsection{Stability}
The solution of {\bf Scheme-I} (see \eqref{scheme}) satisfies the stability criteria~\eqref{es} of the consistent approximation. More precisely, we have the following energy estimates.
 \begin{Theorem}[Stability of {\bf Scheme-I}]\label{Tm1}\hspace{1em}\newline
 Let $(\vrh,\vuh, \vBh)$ be a solution of {\bf Scheme-I}. Then there exist $\xi \in \co{\vrh^{\triangleleft}}{\vrh }$ and  $\zeta \in \co{\vrh^{\rm in}}{\vrh^{\rm out}}$ for any $\sigma \in \facesint$ such that
 \begin{equation} \label{ke}
 \begin{split}
  D_t & \intO{ \left(\frac{1}{2} \vrh  \abs{\auh }^2  + \Hc(\vrh )
 +\frac{1}{2} \abs{\vBh }^2 \right)  }
 +  \mu \norm{\Gradh \vuh }_{L^2} ^2 + \nu \norm{ \Divh \vuh }_{L^2}^2
 +  \norm{\Curlh \vBh }_{L^2}^2
 \\ =&
 - \frac{\TS}{2} \intO{ \vrh^{\triangleleft}\abs{D_t \vuh }^2   }
 - \frac{ \TS}{2} \intO{ \abs{D_t \vBh }^2   }
 - \frac{\TS}2 \intO{ \Hc''(\xi)\abs{D_t \vrh }^2  }
 \\&
 - \sum_{ \sigma \in \facesint } \intG{ \left( \vrh^{\up} \frac{\abs{\us} }{2}
+h^\eps \avg{\vrh}   \right)\abs{\jump{\avc{\vuh}}}^2 }
 -  \sum_{ \sigma \in \facesint } \intG{ \Hc''(\zeta) \jump{  \vrh  } ^2 \left(h^\eps  +\frac{\abs{\us} }{2} \right)}
  \leq 0.
 \end{split}
 \end{equation}
\end{Theorem}
\begin{proof}
First, summing up \eqref{schemeD} and \eqref{schemeM} with the test functions $\phi_h = -\frac{\abs{\auh}^2}{2}$ and $\vh = \vuh $  implies the discrete kinetic energy balance
\begin{equation}\label{ke1}
\begin{aligned}
&\intO{D_t \left(\frac12 \vrh  \abs{\auh}^2 \right) } +  \mu   \norm{\Gradh \vuh }_{L^2}^2  + \nu \norm{ \Divh \vuh }_{L^2}^2 - \intO{ p_h  \Divh \vuh  }
\\& \quad
+  \frac{\TS}{2} \intO{ \vrh^{\triangleleft} \abs{D_t \vuh }^2}
 + \sum_{ \sigma \in \facesint } \intG{ \left(\frac12 \vrh^{\up} \abs{\us}
+h^\eps \avg{\vrh}   \right)\abs{\jump{\avc{\vuh}}}^2 }
\\ &
=  \intO{ ( \Curlh \vBh  \times  \vBh^{\triangleleft} ) \cdot \vuh }
=- \intO{ \Curlh \vBh   \cdot (\vuh  \times \vBh^{\triangleleft} ) }
.
\end{aligned}
\end{equation}
where we have used \eqref{fluxes} and the following equality
\[
 D_t (\vrh \auh)  \cdot \vuh- D_t \vrh \frac{|\auh|^2}{2}
= D_t \Big(\frac12 \vrh  |\vuh |^2 \Big)   + \frac{\TS}2 \vrh^{\triangleleft}  |D_t\vuh |^2 .
 \]

Next, by setting $\vCh =\vBh $ in \eqref{schemeB}, we derive
\begin{equation}\label{ke2}
\intOB{ ( \vuh \times \vBh^{\triangleleft} ) \cdot   \Curlh \vBh  - \alpha | \Curlh  \vBh |^2 }
=\intO{ D_t \vBh  \cdot  \vBh  }
=\intOB{ D_t \frac{|\vBh|^2}{2}  + \frac{\TS}2 |D_t \vBh|^2    }
\end{equation}

Upon setting $b=\Hc(\vr)$ in the renormalized continuity equation~\eqref{r1} and noticing the equality $\vr \Hc'(\vr) - \Hc (\vr)=p(\vr)$, we obtain the balance of internal energy
\begin{equation}\label{ke3}
\intO{ \left( D_t \Hc(\vrh )  - p_h  \Divh \vuh   \right) }
=
- \frac{\TS}2 \intO{\Hc''(\xi)|D_t \vrh |^2  }
-  \sum_{ \sigma \in \facesint } \intG{ \Hc''(\zeta) \jump{  \vrh  } ^2 \left(h^\eps  + \frac12 |\us| \right) }.
\end{equation}
where $\xi \in \co{\vrh^{\triangleleft}}{\vrh }$ and  $\zeta \in \co{\vrh^{\rm in} }{\vrh^{\rm out} }$  are the same as in the renormalized continuity equation~\eqref{r1}.

Finally, we finish the proof by summing up the identities \eqref{ke1}--\eqref{ke3}.
\end{proof}
\medskip

\paragraph{Uniform bounds.} As a consequence of the energy estimates~\eqref{ke}
and  Sobolev's inequality,  we deduce the following  bounds.
\begin{align}
 \label{est_B}
\norm{\vuh}_{L^{2}L^6}  \aleq \norm{\Gradh \vuh}_{L^2 L^2}   \aleq 1,  \;
 \TS^{1/2} \norm{D_t \vBh}_{L^{\infty}L^2}  \aleq 1, \;
\norm{\vBh}_{L^{\infty}L^2}  \aleq 1, \;
\norm{\Curlh \vBh}_{L^{2}L^2}  \aleq 1.
\end{align}
\subsubsection{Consistency}\label{sec_Consistency}
Another step towards the consistent approximation is the consistency.
The numerical solution of {\bf Scheme-I} satisfies the consistency criteria (\ref{cP}) of a consistent approximation. More precisely, we have the following consistency formulation.
 \begin{Theorem}[Consistency of the {\bf Scheme-I}]\label{Tm2} \hspace{1em}\newline
 Let $(\vrh, \vu_h,\vBh)$ be a solution of the discrete problem \eqref{scheme} on the time interval $[0,T]$ with $\TS\approx h$ and $\gamma > \frac{4d}{1+3d}$.
 Then there exists some positive constant $\beta$ such that
\begin{subequations}\label{AcP}
 \begin{equation} \label{AcP1}
 \int_0^\tau \intO{ \left[ \vrh \partial_t \phi + \vrh \vuh \cdot \Gradh \phi \right]} \dt
= -  \intO{ \vrh^0 \phi(0,\cdot) }  +
\order(h^{\beta}) 
 \end{equation}
 for any $\phi \in C_c^2([0,T) \times \Ov{\Omega})$;
 \begin{equation} \label{AcP2}
 \begin{split}
  &
\intTO{ \left[ \vrh \auh \cdot \partial_t \vv + \vrh \auh \otimes \vuh  : \Grad \vv  + p_h \Div \vv \right]}
  -   \intTO{  \S( \Gradh \vuh) : \Grad \vv}
 \\& +  \intTO{ (\Curlh \vBh \times \vBh ) \cdot  \vv}
=-\intO{ \vrh^0 \avc{\vuh^0} \cdot \vv(0,\cdot) }  + \order(h^{\beta}) 
 \end{split}
 \end{equation}
 for any $\vv \in C^2_c([0,T) \times {\Omega}; \R^d)$;
 \begin{equation} \label{AcP3}
 \intO{ \vBh^0 \cdot \vC(0,\cdot)}
 +   \intTOB{ \vBh \cdot \pdt \vC
   -  \alpha  \Curlh \vBh \cdot \Curl\vC +( \vuh \times \vBh )\cdot \Curl\vC}
= \order(h^{\beta}) 
 \end{equation}
  for any $\vC \in C_c^2([0,T)\times \overline{\Omega};\R^d ),\,\vC \times \vn|_{\p \Omega}=0$;
 \begin{equation}\label{AcP4}
 \intO{\vBh \cdot \Grad \psi} =\order(h)
 \end{equation}
 for any $\psi \in C^2( \overline{\Omega}) \cap L^2_0(\Omega)$.
\end{subequations}
 \end{Theorem}
\begin{proof}
First, recalling \cite[Theorem 13.2]{FeLMMiSh} we know \eqref{AcP1} holds and  there exists a $\beta>0$ such that
\begin{equation}\label{Ac1}
\begin{aligned}
 & \intTO{ D_t (\vrh  \auh) \cdot  \vh   } -
 \int_0^T \sum_{ \sigma \in \facesint } \intG{  \Fup(\vrh   \auh,\vuh  ) \cdot
\jump{\avh}   } \dt
 \\&\quad + \mu \intTO{  \Gradh \vuh  : \Gradh \vh  }
 +  \intTO{( \nu \Divh \vuh  -p_h ) \Divh \vh  }
 \\ &= - \intO{ \vrh^0 \avc{\vuh^0} \cdot \vv(0) }  -
\int_0^T \intO{ \left[ \vrh \auh \cdot \partial_t \vv + \vrh \auh \otimes \vuh  : \Grad \vv  + p_h \Div \vv \right]} \dt
\\& \quad
 +   \int_0^T \intO{  \S( \Gradh \vuh) : \Grad \vv}  \dt
  + h^{\beta}
.
\end{aligned}
\end{equation}

Then we derive \eqref{AcP2} by combining \eqref{Ac1} with the following estimates
\begin{align*}
 &\abs{ \intTOB{ ( \Curlh \vBh \times \vBh^{\triangleleft}  ) \cdot \PiV\vv  - ( \Curlh \vBh \times \vBh  ) \cdot \vv} }
\\&= \abs{
 \intTOB{ (  \Curlh \vBh \times \vBh  ) \cdot (\PiV\vv -\vv)  +
 (\Curlh \vBh \times (\vBh-\vBh^{\triangleleft})) \cdot\vv}  }
\\&
 \aleq \norm{\Curlh \vBh}_{L^2L^2}
 \left( \norm{\vBh}_{L^2L^2} h \norm{\vv}_{C^1}
 + \TS \norm{D_t \vBh}_{L^2L^2} \norm{\vv}_{C^0}
 \right)
\aleq h + \TS^{1/2} \aleq h^{1/2}.
\end{align*}
where we have used H\"older's inequality, the uniform bounds \eqref{est_B} as well as the interpolation estimate~\eqref{IE}.
We are left with the proof of \eqref{AcP3} and \eqref{AcP4}. To proceed, we set $\vCh =\PiN \vC$ as the test function in \eqref{schemeB} and analyze each term in the following. First, for the time derivative term we have
\[ \begin{aligned}
&\int_0^T  \intO{D_t \vBh \cdot \PiN \vC}
=\frac{1}{\TS}\int_0^T  \intO{ \vBh(t)\cdot \PiN \vC(t)}
- \frac{1}{\TS}\int_{-\TS}^{T-\TS}  \intO{ \vBh(t)\cdot \PiN \vC(t+\TS)}
\\& =
-\int_0^T  \intO{ \vBh(t)\cdot D_t  \PiN \vC }
- \frac{1}{\TS}\int_{-\TS}^{0}  \intO{ \vBh(t)\cdot \PiN \vC(t+\TS)}
\\ & \quad + \frac{1}{\TS}\int_{T-\TS}^{T}  \intO{ \vBh(t)\cdot \PiN \underbrace{\vC(t+\TS)}_{=0}}
\\& =
-\int_0^T  \intO{ \vBh(t)\cdot D_t  \PiN \vC }
-  \intO{ \vBh^0\cdot \int_0^{\TS}\PiN \vC(t)\dt }
\\& =
-\int_0^T  \intO{ \vBh(t)\cdot \pdt \vC }
-  \intO{ \vBh^0\cdot \vC(0)}
+I_1+I_2,
\end{aligned}
\]
where
\[I_1 = \int_0^T  \intO{ \vBh(t)\cdot (\pdt \vC - D_t  \PiN \vC) }\dt, \quad
I_2 =   \intO{ \vBh(0)\cdot \left(\vC(0) - \int_0^{\TS}\PiN \vC(t)\dt \right)}.
\]
By H\"older's inequality and the estimates \eqref{est_B} we have
\[
\abs{ I_1 }
\aleq  \norm{\vBh}_{L^2L^2} \TS \norm{\vC}_{C^2} \aleq h, \quad
\abs{ I_2 }
\aleq  \norm{\vBh^0}_{L^1} \TS \norm{\vC}_{C^2} \aleq h.
\]
Next, using H\"older's inequality again with the uniform bounds \eqref{est_B} and interpolation estimate~\eqref{IE} we derive
\[
 \int_0^T  \intO{ \Curlh \vBh  \cdot ( \Curlh \PiN \vC -  \Curl \vC)}\dt
 \aleq  h\norm{\vC}_{C^2} \norm{\Curlh \vBh}_{L^2L^2}
 \aleq h,
\]
and
\[
\begin{aligned}
& \intTOB{  (  \vuh\times \vBh^{\triangleleft} ) \cdot \Curlh \PiN \vC - ( \vuh\times \vBh )   \cdot \Curl \vC}
\\& =
 \intTO{ ( \vuh\times \vBh^{\triangleleft} ) \cdot (\Curlh \PiN \vC -  \Curl \vC)}
+   \intTO{ (\vuh\times  (\vBh^{\triangleleft} -\vBh)  ) \cdot \Curl \vC }
\\& \aleq   \norm{\vuh}_{L^2L^6} \left( h \norm{\vBh}_{L^\infty L^2} \norm{\vC}_{C^2} +    \TS  \norm{D_t \vBh}_{L^\infty L^2} \norm{\vC}_{C^1}  \right)
 \aleq h + \TS^{1/2} \aleq h^{1/2}.
\end{aligned}
\]
Consequently, summing up the above terms finishes the proof of \eqref{AcP3}.
Finally, concerning the proof of \eqref{AcP4}, we recall Item 5 of Lemma~\ref{lem_b} to deduce
\[
\intO{\vBh\cdot \Grad \psi} = \intO{\vBh\cdot \Grad(\psi-\PiW\psi)} \aleq \norm{\vBh}_{L^\infty L^2} h \norm{\psi}_{C^2} \aleq h,
\]
which completes the proof.
\end{proof}
\subsubsection{Convergence}
Now we are ready to prove the convergence of {\bf Scheme-I}.
\begin{Theorem}[Convergence of {\bf Scheme-I}]\label{Th_SA}\hspace{1em}\newline
Let $(\vrh, \vuh, \vBh)$ be a solution to {\bf Scheme-I}  with $\TS \approx h$ and $\gamma > \frac{4d}{1+3d}$.  Then it converges in the sense of Theorem~\ref{Th2}.
\end{Theorem}
\begin{proof}
Note that the compatibility of the discrete differential operators has been presented in \cite[Section 11.4 and Section 13.4]{FeLMMiSh}.
Combing Theorem~\ref{Tm1} and Theorem~\ref{Tm2} we conclude that the numerical solution of {\bf Scheme-I} is a consistent approximation of the MHD system in the sense of Definition~\ref{def_ca}.
Applying Theorem~\ref{Th2} we derive the convergence for {\bf Scheme-I}.
\end{proof}

\subsection{Example-II}\label{oth_num}
In this example we introduce {\bf Scheme-II} on a periodic domain identified with the flat torus.
On one hand,  we use the same discretization as {\bf Scheme-I} for the magnetic field. On the other hand,  we follow Feireisl et al.~\cite{FLMS_FVNS} with piecewise constant discretizations for  the approximation of the density, velocity, and pressure for the Navier-Stokes part.
\begin{tcolorbox}
{\bf Scheme-II.}\label{def_fv}
Let $\Omega= \mathbb{T}^d =\left( [0,1]|_{\{0,1\}}\right)^d$ and $p$ satisfy \eqref{plaw}. Given the initial data \eqref{ini_c} we set  $(\vrh^0,\vuh^0, \vBh^0) =(\PiQ \vr_0, \PiV \vu_0, \PiN \vB_0)$  and  seek $(\vrh , \vuh  , \vBh ) \in L_{\TS}(0,T; Q_h \times(Q_h)^d \times \Nh)$ such that
 \eqref{schemeB} holds for any $\vCh \in \Nh$ and
\begin{equation*}
D_t \vr  _K + \sum_{\sigma \in \facesK} \frac{|\sigma|}{|K|} \Fup(\vrh ,\vuh ) =0,\quad \mbox{ for all } K\in \grid;
\end{equation*}
\begin{equation*}
\begin{aligned}
& D_t (\vrh  \vuh )_K + \sum_{\sigma \in \facesK} \frac{|\sigma|}{|K|}
\left( \Fup(\vrh  \vuh ,\vuh )  + \avg{p_h } \vc{n}
- \mu \frac{\jump{\vuh }}{d_\sigma}
-(\mu + \lambda) \avg{\Divh \vuh } \vc{n}\right)
\\&  =- \frac{1}{|K|}\intK{\Curlh \vBh  \times \vBh^{\triangleleft}},
\quad \mbox{ for all } K\in \grid;
\end{aligned}
\end{equation*}
where the artificial diffusion parameter $\eps$ satisfies
\begin{equation*}
\eps >0 \mbox{ if } \gamma \geq 2 \quad \mbox{ and }\quad   \eps \in(0, 2 \gamma-1 -d/3) \mbox{ if } \gamma \in(1,2).
\end{equation*}
\end{tcolorbox}
Here, the discrete operators  $\jump{\cdot}$, $\avg{\cdot}$ and the numerical flux $\Fup$ are defined in \eqref{op_diff} and \eqref{num_flux}, $\grid$ is a uniform structured mesh discretization of $\Omega$ consisting of rectangles in 2D or cuboids in 3D,  $d_\sigma = h$ denotes the distance between the centers of neighboring elements.  Moreover, the discrete divergence operator for the piecewise constant velocity $\vuh \in (Q_h)^d$ is given by
\begin{equation*}
(\Divh \vuh)_K =
\frac{1}{|K|}\sum_{\sigma\in \facesK}|\sigma| \avg{\vuh} \cdot \vc{n} \quad \forall \; K \in \grid.
\end{equation*}
\begin{Remark}
The difference between {\bf Scheme-I} and {\bf Scheme-II} mainly relies in the discretization of the Navier-Stokes part.
It is analogous to check that {\bf Scheme-II} also satisfies all the properties stated in Lemma~\ref{lem_b}, e.g. conservation of mass, positivity of density, weakly divergence free of magnetic field.
Further, noticing that the stability and consistency of the Navier-Stokes part of  {\bf Scheme-II}  have been analyzed in \cite[Chapter 11]{FeLMMiSh} with $\gamma>1$ and $\TS \approx h$, we may analogously show that the numerical solution of {\bf Scheme-II} is a consistent approximation of the MHD system in the sense of Definition~\ref{def_ca}. Systematically, we have the following convergence result.
\end{Remark}
\begin{Proposition}[Convergence of {\bf Scheme-II}]\label{thm_cons}
Let $(\vrh, \vuh, \vBh)$ be a solution to {\bf Scheme-II}  with $\TS \approx h$ and $\gamma >1$.  Then it converges in the sense of Theorem~\ref{Th2}.
\end{Proposition}

\section{Conclusion}\label{sec_end}

We introduced the concept of DW solution and consistent approximation for multi-dimensional compressible MHD system~\eqref{pde}--\eqref{ini_c}.
We derived the weak--strong uniqueness property for the DW solution, meaning that the DW solution coincides with the classical solution (emanating from the same initial data) as long as the latter exists.
Further, we proved the convergence of the consistent approximation towards the DW solution as well as the classical solution on the lifespan of the latter. Interpreting the consistent approximation as the stability and consistency of a numerical solution, we established a generalized Lax equivalence  theory. Finally, we applied this theory for the convergence analysis of two mixed finite volume--finite element methods. These two methods preserve the conservation of mass, positivity of density, stability of total energy, and weakly divergence free of the magnetic field.

\bigskip

\clearpage \newpage

\centerline{\bf \large{Acknowledgements}}

The research of Y.~Li is supported by National Natural Science Foundation of China under grant No. 12001003. The research of B.~She is supported by Czech Science Foundation,  Grant Agreement 21-02411S. The institute of Mathematics of the Czech Academy of Sciences is supported by RVO:67985840.

%
%

\def\cprime{$'$} \def\ocirc#1{\ifmmode\setbox0=\hbox{$#1$}\dimen0=\ht0
  \advance\dimen0 by1pt\rlap{\hbox to\wd0{\hss\raise\dimen0
  \hbox{\hskip.2em$\scriptscriptstyle\circ$}\hss}}#1\else {\accent"17 #1}\fi}

\bibliography{citace}
\bibliographystyle{siamplain}

\end{document}